\pdfoutput=1
\documentclass[letterpaper]{amsart}
\usepackage{lmodern}

\usepackage[T1]{fontenc}
\usepackage{amssymb}
\usepackage{enumitem}
\usepackage{mathtools}

\usepackage{tikz-cd}
\usepackage[pdfusetitle]{hyperref}
\usepackage{cleveref}

\usepackage{mathbbol}
\DeclareSymbolFontAlphabet{\mathbbm}{bbold}
\DeclareSymbolFontAlphabet{\mathbb}{AMSb}

\tikzcdset{arrow style=Latin Modern}

\mathtoolsset{mathic}
\DeclareMathAlphabet\mathbfit{OML}{cmm}{b}{it}

\let\setminus\smallsetminus

\newlist{enumarabic}{enumerate}{1}
\setlist[enumarabic]{font=\normalfont,label=(\arabic*),leftmargin=0.3in}
\newlist{enumroman}{enumerate}{1}
\setlist[enumroman]{font=\normalfont,label=(\roman*),leftmargin=0.3in}

\numberwithin{equation}{section}
\allowdisplaybreaks[4]

\theoremstyle{plain}
\newtheorem{theorem}{Theorem}[section]
\newtheorem{proposition}[theorem]{Proposition}
\newtheorem{lemma}[theorem]{Lemma}
\newtheorem{corollary}[theorem]{Corollary}

\theoremstyle{definition}

\newtheorem{remark}[theorem]{Remark}
\newtheorem{example}[theorem]{Example}

\theoremstyle{remark}
\newtheorem*{acknowledgements}{Acknowledgements}

\let\newterm\emph

\def\arxiv#1{\href{http://arxiv.org/abs/#1}{\texttt{arXiv:#1}}}
\def\doi#1{\href{http://doi.org/#1}{doi:#1}}

\def\cf{\emph{cf.}}

\def\N{\mathbb N}
\def\R{\mathbb R}
\def\kk{\Bbbk}

\let\epsilon\varepsilon
\let\phi\varphi
\let\emptyset\varnothing

\DeclareMathOperator{\Hom}{Hom}
\def\id{\mathrm{id}}

\let\shuffle\nabla
\def\Simp{\mathbbm{\Delta}}
\def\II{\mathbb{I}}
\def\TopI{I}
\let\Cube\square
\def\sSet{\mathrm{sSet}}
\DeclareMathOperator{\TT}{\mathbfit{T}}
\def\tt{\mathbfit{t}}
\def\ff{q}
\def\FF{F}
\def\FFF{f}
\def\PSI{p}
\def\OM{\boldsymbol{\Omega}}

\def\Deltabar{\bar\Delta}
\def\SS{\mathbfit{S}}
\def\ii{\mathbfit{i}}
\def\jj{\mathbfit{j}}
\def\kkk{\mathbfit{k}}
\def\tSz{t_{\mathrm{Sz}}}
\def\fSz{f_{\mathrm{Sz}}}
\DeclareMathOperator{\Sz}{Sz}
\DeclareMathOperator{\Shuff}{Shuff}
\def\susp{\mathbf{s}}
\def\desusp{\susp^{-1}}
\def\deg#1{|#1|}

\def\TWO{\mathbf{2}}

\def\CobarEl#1{\langle#1\rangle}
\def\bigCobarEl#1{\bigl\langle#1\bigr\rangle}

\def\bbOM{\mathchoice{\scalebox{1.05}{$\displaystyle\mathbbm{\Omega}$}}%
  {\scalebox{1.05}{$\textstyle\mathbbm{\Omega}$}}%
  {\scalebox{1.05}{$\scriptstyle\mathbbm{\Omega}$}}%
  {\scalebox{1.05}{$\scriptscriptstyle\mathbbm{\Omega}$}}}

\begin{document}

\title[The Szczarba map and the cubical cobar construction]{The Szczarba map and\\the cubical cobar construction}
\author{Matthias Franz}
\thanks{The author was supported by an NSERC Discovery Grant.}
\address{Department of Mathematics, University of Western Ontario,
  London, Ont.\ N6A\;5B7, Canada}
\email{mfranz@uwo.ca}


\begin{abstract}
  We consider a twisting function from a \(1\)-reduced simplicial set~\(X\) to a simplicial group~\(G\).
  We prove in detail that the associated Szczarba operators induce a simplicial map from the triangulation
  of the cubical cobar construction of~\(X\) to~\(G\). This confirms a result due to
  Minichiello--Rivera--Zeinalian and gives, as pointed out by these authors,
  a conceptual proof of the fact that the dga map~\(\OM\,C(X) \to C(G)\)
  induced by Szczarba's twisting cochain is comultiplicative.
\end{abstract}

\maketitle

\section{Introduction}

Let \(X\) be a \(1\)-reduced simplicial set, and let \(\tau\colon X_{>0}\to G\) be a twisting function.
An important example is the universal twisting function~\(\tau_{X}\colon X_{>0}\to GX\), where \(GX\) is the Kan loop group of~\(X\).
Based on~\(\tau\), Szczarba~\cite{Szczarba:1961} defined an explicit twisting cochain
\begin{equation}
  \tSz\colon C(X)\to C(G),
\end{equation}
which gives rise to a morphism of differential graded algebras (dgas)
\begin{equation}
  \label{eq:Omega-CX-CG}
  \fSz\colon\OM\,C(X) \to C(G),
  \qquad
  \CobarEl{x_{1}|\dots|x_{k}} \mapsto \tSz(x_{1})\cdots \tSz(x_{k}).
\end{equation}
Here \(C(-)\) denotes normalized chains with coefficients in some commutative ring~\(\kk\), and \(\OM\,C(X)\) is the reduced cobar construction.

The \newterm{Szczarba map}~\(\fSz\) has the following form:
An \((n+1)\)-simplex~\(x\in X\) defines an element~\(\CobarEl{x}\) of degree~\(n\) in~\(\OM\,C(X)\).
Its image~\(\fSz(\CobarEl{x})=\tSz(x)\in C(G)\) is, for~\(n\ge1\), the signed sum of certain \(n\)-simplices~\(\Sz_{\ii}(x)\in G\).
(For~\(n=0\) the formula is slightly modified, see~\eqref{eq:szczarba-twisting-cochain}.)
The \newterm{Szczarba operators}~\(\Sz_{\ii}\colon X_{n+1}\to G_{n}\) are indexed
by a set~\(\SS_{n}\) of \(n!\)~integer sequences~\(\ii\) of length~\(n\).

Now both \(C(G)\) and~\(\OM\,C(X)\) are not only dgas, but also differential graded coalgebras (dgcs);
the diagonal of~\(\OM\,C(X)\) is defined in terms of so-called homotopy Gerstenhaber operations on~\(C(X)\).
In~\cite{Franz:szczarba2} the present author has shown that the map~\eqref{eq:Omega-CX-CG} is compatible with the dgc structures,
so that \eqref{eq:Omega-CX-CG} becomes a morphism of dg~bialgebras.
In the more general setting of reduced simplicial sets,
Minichiello--Rivera--Zeinalian~\cite[Cor.~1.4]{MinichielloRiveraZeinalian:2023}
have presented the following conceptual explanation of these facts:

Out of the simplicial set~\(X\) one can define a cubical monoid~\(\bbOM X\),
called the cubical cobar construction. The cubes in~\(\bbOM X\) correspond to strings of simplices in~\(X\) decorated with formal degenerations.
The normalized cubical chain complex~\(C(\bbOM X)\) turns out to be isomorphic to~\(\OM\,C(X)\) as a dg~bialgebra, \cf~\Cref{thm:iso-cobars}.
In light of this isomorphism, we write the \(n\)-cube in~\(\bbOM X\) determined by an \((n+1)\)-simplex~\(x\in X\)
again as~\(\CobarEl{x}\).

Let \(\TT \bbOM X\) be the formal triangulation of the cubical set~\(\bbOM X\). It is a simplicial set;
each non-degenerate \(n\)-cube in~\(\bbOM X\) corresponds to \(n!\)~non-degenerate \(n\)-simplices in~\(\TT \bbOM X\),
see \Cref{sec:triangulate-cubical-sets}.
According to~\cite{MinichielloRiveraZeinalian:2023},
the \(n!\)~terms~\(\Sz_{\ii}(x)\) appearing in~\(\fSz(\CobarEl{x})\)
are the images of the \(n!\)~simplices in~\(\TT \bbOM X\) corresponding
to the \(n\)-cube~\(\CobarEl{x}\) under a simplicial map~\(\FFF\colon\TT \bbOM X\to G\).
Hence \(\fSz\) factors as
\begin{equation}
  \OM\,C(X) \cong C(\bbOM X) \stackrel{\tt}{\longrightarrow} C(\TT \bbOM X) \stackrel{C(\FFF)}{\longrightarrow} C(G).
\end{equation}
The isomorphism on the left is one of dg~bialgebras.
The map~\(\tt\) in the middle is essentially the shuffle map, which is known to be
a quasi-isomorphism of dg~bialgebras in this case, see \Cref{thm:triangulation-product-group}.
The rightmost map is so, too, because it is induced from a morphism of simplicial monoids.
Hence \eqref{eq:Omega-CX-CG} is in particular a dgc map. QED.

Clearly, not every assignment of a simplex in~\(G\) to each simplex in the triangulation
of a cube in~\(\bbOM X\) will define a simplicial map~\(\TT\bbOM X\to G\). The assignment must respect
both the cubical structure of~\(\bbOM X\) and the simplicial structure of the triangulation of each cube.
After reading \cite{MinichielloRiveraZeinalian:2023}, the present author was still unsure
as to which properties of the Szczarba operators ensure that this construction is actually well-defined.
The present paper may thus be considered as the author's attempt to work out the details.
This is inverse to the approach taken by Minichiello--Rivera--Zeinalian in~\cite{MinichielloRiveraZeinalian:2023}
and further elaborated on 
in~\cite{MinichielloRiveraZeinalian:szczarba}.
Starting from categorical considerations, these authors define a simplicial map~\(\TT\bbOM X\to G\)
and then show that it induces Szczarba's twisting cochain.
In the present paper, we instead take the Szczarba operators as our starting point and prove that the map~\(\TT\bbOM X\to G\)
induced by them is simplicial.

We start in \Cref{sec:prelim} by reviewing
the triangulation of a cubical set in detail and setting up several bijections to translate between permutations
and Szczarba's indexing scheme. It then turns out that everything hinges on two results:
The first one, contained in Szczarba's original paper and later highlighted by Hess--Tonks~\cite{HessTonks:2006},
describes how the Szczarba operators interact with face operators, see \Cref{thm:szczarba-d}.
The second one (\Cref{thm:szczarba-s}) is about their interaction with degeneracy operators;
this formula was established by the present author in~\cite{Franz:szczarba2}.
As a consequence, we can confirm Minichiello--Rivera--Zeinalian's result in our context.

\begin{theorem}
  \label{thm:intro:main}
  Let \(X\) be a \(1\)-reduced simplicial set and \(G\) a simplicial group,
  and let \(\tau\colon X_{>0}\to G\) be a twisting function.
  The Szczarba operators~\(\Sz_{\ii}\) induce a morphism of simplicial monoids
  \begin{equation*}
    \FFF\colon\TT\bbOM X\to G.
  \end{equation*}
  Moreover, the composition
  \begin{equation*}
    \OM\,C(X) \cong C(\bbOM X) \stackrel{\tt}{\longrightarrow} C(\TT \bbOM X) \stackrel{C(\FFF)}{\longrightarrow} C(G),
  \end{equation*}
  is the Szczarba map~\(\fSz\) induced by Szczarba's twisting cochain~\(\tSz\).
\end{theorem}

A more precise version will be given in \Cref{thm:main}.

\begin{corollary}
  The map~\(\fSz\colon\OM\,C(X) \to C(G)\) is a morphism of dg~bialgebras.
\end{corollary}

We conclude this introduction by mentioning a recent preprint of Cai~\cite{Cai:twisting}
(which has appeared after the first version of the present paper).
Cai constructs a twisting cochain based on a careful analysis of
Berger's work on simplicial prisms and loop spaces \cite{Berger:1995}.
While Cai's twisting cochain in similar in structure to Szczarba's,
it does not involve the group inversion on~\(G\).
As a result, Cai's approach works more generally for simplicial monoids,
which answers a question raised by Rivera~\cite{Rivera:email}.

Given a simplicial \(G\)-set~\(F\), one can form the twisted Cartesian product~\(X\times_{\tau}F\)
(the simplicial analogue of an associated bundle) as well as the twisted tensor product~\(C(X) \otimes_{t_{Sz}} C(F)\),
\cf~\cite[Secs.~2.4~\&~8]{Franz:szczarba2}. In addition to the map~\(\fSz\), Szczarba introduced a twisted shuffle map
\begin{equation}
  \psi\colon C(X) \otimes_{t_{Sz}} C(F) \to C(X\times_{\tau}F)
\end{equation}
and showed that it is a quasi-isomorphism of complexes \cite[Thm.~2.4]{Szczarba:1961}.
The homotopy Gerstenhaber structure of~\(C(X)\) also induces a dgc structure on the twisted tensor product,
see again \cite[Sec.~8]{Franz:szczarba2}.
In~\cite[Thm.~1.3]{Franz:szczarba2} the current author showed, again via a direct computation,
that \(\psi\) is actually a morphism of dgcs.
Minichiello--Rivera--Zeinalian~\cite{MinichielloRiveraZeinalian:2023} do not consider
the twisted shuffle map. Cai shows that his version of~\(\psi\) is a chain map;
the dgc structure is not discussed in~\cite{Cai:twisting}.

\begin{acknowledgements}
  I am grateful to Jim Stasheff for several discussion during the preparation of this work.
  I also thank Chris Kapulkin and Manuel Rivera for comments on an earlier version of this paper,
  and Li Cai for explaining his work to me.
\end{acknowledgements}

\section{Preliminaries}
\label{sec:prelim}

Throughout this paper, we fix a commutative ring~\(\kk\) with unit.

\subsection{Permutations}

Let \(\pi_{1}\in S_{n_{1}}\),~\ldots,~\(\pi_{m}\in S_{n_{m}}\) be permutations and set \(n=n_{1}+\dots n_{m}\).
The \newterm{concatenation}~\(\pi_{1}\sqcup\dots\sqcup\pi_{m}\in S_{n}\) is the induced permutation
on the disjoint union of their index sets, relabelled as~\(\{1,\dots,n\}\) in increasing order.
For example,
\begin{equation}
  \begin{pmatrix}
    1 & 2 \\ 2 & 1
  \end{pmatrix}
  \sqcup
  \begin{pmatrix}
    1 & 2 & 3 \\ 3 & 1 & 2
  \end{pmatrix}
  \sqcup
  \begin{pmatrix}
    1 \\ 1
  \end{pmatrix}
  =
  \begin{pmatrix}
    1 & 2 & 3 & 4 & 5 & 6 \\
    2 & 1 & 5 & 3 & 4 & 6
  \end{pmatrix}.
\end{equation}

\begin{lemma}
  \label{thm:Psi-k-l}
  For any~\(k\),~\(l\ge0\) there is a bijection
  \begin{equation*}
    \Psi_{k,l}\colon \Shuff(k,l)\times S_{k}\times S_{l} \to S_{k+l},
    \quad
    \bigl((\alpha,\beta),\sigma,\tau\bigr) \mapsto (\sigma\sqcup\tau)\circ\pi_{\alpha,\beta},
  \end{equation*}
  where
  \begin{equation*}
    \pi_{\alpha,\beta} = \begin{pmatrix}
      \alpha_{1} & \cdots & \alpha_{k} & \beta_{1} & \cdots & \beta_{l} \\
      1 & \cdots & k & k+1 & \cdots & k+l
    \end{pmatrix} \in S_{k+l}.
  \end{equation*}
\end{lemma}

\begin{proof}
  For the inverse map, let \(\alpha\) be the increasing sequence of all~\(j\in\{1,\dots,k+l\}\) such that \(\pi(j)\in\{1,\dots,k\}\)
  and likewise \(\beta\) the one with~\(\pi(j)\in\{k+1,\dots,k+l\}\).
  The permutations~\(\sigma\) and~\(\tau\) encode the order in which the elements of~\(\alpha\) and~\(\beta\)
  appear in~\(\pi\).
\end{proof}

\begin{example}
  The permutation
  \begin{equation}
    \pi = \begin{pmatrix}
      1 & 2 & 3 & 4 & 5 \\
      3 & 2 & 5 & 4 & 1
    \end{pmatrix}
  \end{equation}
  corresponds under~\(\Psi_{2,3}\) to
  \begin{equation}
    \alpha = (2, 5),
    \quad
    \beta = (1, 3, 4),
    \quad
    \sigma = \begin{pmatrix}
      1 & 2 \\
      2 & 1
    \end{pmatrix},
    \quad
    \tau = \begin{pmatrix}
      1 & 2 & 3 \\
      1 & 3 & 2
    \end{pmatrix}.
  \end{equation}
\end{example}

\subsection{Simplicial sets}

For the convenience of the reader, we recall the definition, \cf~\cite[Def.~1.1]{May:1968}.
A \newterm{simplicial set} is an \(\N\)-graded set~\(X\) with face operators
\begin{equation}
  d_{i}\colon X_{n}\to X_{n-1} \qquad\text{for~\(0\le i\le n\) where \(n\ge 1\)}
\end{equation}
and degeneracy operators
\begin{equation}
  s_{i}\colon X_{n}\to X_{n+1} \qquad\text{for~\(0\le i\le n\) where \(n\ge 0\)}
\end{equation}
satisfying the simplicial identities
\begin{align}
  d_{i}\,d_{j} &= d_{j-1}\,d_{i} & & \text{if \(0\le i<j\le n\) (\(n\ge 2)\),} \\
  s_{i}\,s_{j} &= s_{j+1}\,s_{i} & & \text{if \(0\le i\le j\le n\),} \\
  d_{i}\,s_{j} &= \begin{cases}
    s_{j-1}\,d_{i} & \text{if \(i<j\),} \\
    \id & \text{if \(j\le i\le j+1\),} \\
    s_{j}\,d_{i-1} & \text{if \(i>j+1\),}
  \end{cases} & & \text{where \(0\le i\le n+1\), \(0\le j\le n\).}
\end{align}
The elements of~\(X_{n}\) are called \(n\)-simplices.
An \(n\)-simplex~\(x\) is degenerate if it is of the form~\(x=s_{i}\,x'\) for some \((n-1)\)-simplex~\(x'\) and some~\(0\le i\le n-1\).

The normalized simplicial chain complex~\(C(X)\) with coefficients in~\(\kk\)
is obtained from the non-normalized complex by dividing out the subcomplex spanned
by the degenerate simplices.
The differential of an \(n\)-simplex~\(x\) is
\begin{equation}
  d\,x = \sum_{i=0}^{n}(-1)^{i}\,d_{i}x
\end{equation}
for~\(n\ge1\) and \(dx=0\) for~\(n=0\). The complex~\(C(X)\) becomes a differential graded coalgebra (dgc)
via the diagonal
\begin{equation}
  \Delta\,x = \sum_{i=0}^{n} x(0\dots i)\otimes x(i\dots n)
\end{equation}
for~\(x\in X_{n}\) and the augmentation sending every \(0\)-simplex to~\(1\in\kk\).
Above,
\begin{equation}
  x(0\dots i) = d_{i+1}\cdots d_{n}\,x
  \qquad\text{and}\qquad
  x(i\dots n+1) = d_{0}\cdots d_{i-1}\,x
\end{equation}
denote, respectively, the \(i\)-dimensional front face and \((n-i)\)-dimensional back face of~\(x\).

Let \(X\) and~\(Y\) be simplicial set, and
let \(x\in X\) be a \(k\)-simplex and \(y\in Y\) an \(l\)-simplex where \(k\),~\(l\ge0\).
Set \(n=k+l\). The shuffle map 
\begin{equation}
  \label{eq:shuffle-map}
  \shuffle\colon C(X)\otimes C(Y) \to C(X\times Y)
\end{equation}
sends \(x\otimes y\) to the signed sum of \(n\)-simplices
\begin{equation}
  \label{eq:def-ez}
  \sum_{\!\!\!\!(\alpha,\beta)\in\Shuff(k,l)\!\!\!\!}
  (-1)^{(\alpha,\beta)} \bigl(s_{\beta_{l}-1}\cdots s_{\beta_{1}-1}x, s_{\alpha_{k}-1}\cdots s_{\alpha_{1}-1} y\bigr)
\end{equation}
where \(\Shuff(k,l)\) denotes the set of \((k,l)\)-shuffles, that is, the partitions
of~\(\{1,\dots,n\}\) into two sets~\(\alpha=\{\alpha_{1}<\dots<\alpha_{k}\}\) and~\(\beta=\{\beta_{1}<\dots<\beta_{l}\}\)
of size~\(k\) and~\(l\), respectively.\footnote{One often takes partitions of the set~\(\{0,\dots,n-1\}\) instead,
so that there is no need for subtracting \(1\)'s in formula~\eqref{eq:def-ez}.}
Moreover, \((-1)^{(\alpha,\beta)}\) denotes the sign of the permutation sending \((1,\dots,n)\)
to~\((\alpha_{1},\dots,\alpha_{k},\beta_{1},\dots,\beta_{l})\).
For simplicity, we write the product simplex appearing in~\eqref{eq:def-ez} as
\begin{equation}
  \label{eq:def-s-alpha-1}
  (s_{\beta-1}x, s_{\alpha-1}y).
\end{equation}
The shuffle map is associative and moreover a morphism of dgcs, see \cite[(17.6)]{EilenbergMoore:1966}.

\begin{example}
  \label{ex:singular-simplex}
  The singular \(n\)-simplices~\(x\colon\Delta^{n}\to Z\) in a topological space~\(Z\)
  form the \(n\)-simplices of a simplicial set.
  Here \(\Delta^{n}=\{\,a\in\R^{n+1}\mid a_{0}+\dots+ a_{n} = 1\,\}\) is the
  standard \(n\)-simplex. The face and degeneracy operators are
  \begin{align}
    (d_{i}x)(a_{0},\dots,a_{n-1}) &= x(a_{0},\dots,a_{i-1},0,a_{i},\dots,a_{n}), \\
    (s_{i}x)(a_{0},\dots,a_{n+1}) &= x(a_{0},\dots,a_{i}+a_{i+1},\dots,a_{n+1}).
  \end{align}

  If \(x\) and~\(y\) are singular simplices,
  then the singular simplex~\((s_{\beta-1}x, s_{\alpha-1}y)\) on the right-hand side of~\eqref{eq:def-ez}
  maps~\(a\in\Delta^{n}\) to
  \begin{multline}
    \bigl(x(a_{0}+\dots+a_{\alpha_{1}-1},a_{\alpha_{1}}+\dots+a_{\alpha_{2}-1},\dots,a_{\alpha_{k}}+\dots+a_{n}), \\
    y(a_{0}+\dots+a_{\beta_{1}-1},a_{\beta_{1}}+\dots+a_{\beta_{2}-1},\dots,a_{\beta_{l}}+\dots+a_{n}) \bigr).
  \end{multline}
\end{example}

\subsection{Cubical sets}
\label{sec:cubical}

A \newterm{cubical set} is an \(\N\)-graded set~\(Y\) with face operators
\begin{equation}
  d^{\epsilon}_{i}\colon Y_{n}\to Y_{n-1} \qquad\text{for~\(\epsilon\in\{0,1\}\) and \(1\le i\le n\)}
\end{equation}
and degeneracy operators
\begin{equation}
  s_{i}\colon Y_{n}\to Y_{n+1} \qquad\text{for~\(1\le i\le n+1\)}
\end{equation}
satisfying the cubical identities
\begin{align}
  d^{\epsilon}_{i}\,d^{\epsilon'}_{j} &= d^{\epsilon'}_{j-1}\,d^{\epsilon}_{i} & & \text{if \(1\le i<j\le n\),\; \(\epsilon\),~\(\epsilon'\in\{0,1\}\),} \\
  \label{eq:cube-ss}
  s_{i}\,s_{j} &= s_{j+1}\,s_{i} & & \text{if \(1\le i\le j\le n+1\),} \\
  d^{\epsilon}_{i}\,s_{j} &= \begin{cases}
    s_{j-1}\,d^{\epsilon}_{i} & \text{if \(i<j\),} \\
    \id & \text{if \(i=j\),} \\
    s_{j}\,d^{\epsilon}_{i-1} & \text{if \(i>j\),}
  \end{cases} & & \text{where \(1\le i\le n+1\), \(1\le j\le n+1\).}
\end{align}
The elements of~\(Y_{n}\) are called \(n\)-cubes.
An \(n\)-cube~\(y\) is degenerate if it is of the form~\(y=s_{i}\,y'\) for some \((n-1)\)-cube~\(y'\) and some~\(1\le i\le n\).

A \newterm{cubical set with connections} additionally has connection operators
\begin{equation}
  \gamma_{i}\colon Y_{n}\to Y_{n+1} \qquad\text{for~\(1\le i\le n\) where \(n\ge 1\)}
\end{equation}
satisfying
\begin{align}
  \gamma_{i}\,\gamma_{j} &= \gamma_{j+1}\,\gamma_{i} \qquad\text{if \(i\le j\),} \\
  \shortintertext{where \(i\),~\(j\in\{1,\dots,n\}\),}
  d^{\epsilon}_{i}\,\gamma_{j} &= \begin{cases}
    \gamma_{j-1}\,d^{\epsilon}_{i} & \text{if \(i<j\),} \\
    s_{j}\,d^{\epsilon}_{j} & \text{if \(i\in\{j,j+1\}\) and \(\epsilon=0\),} \\
    \id & \text{if \(i\in\{j,j+1\}\) and \(\epsilon=1\),} \\
    \gamma_{j}\,d^{\epsilon}_{i-1} & \text{if \(i>j+1\),}
  \end{cases} \\
  \shortintertext{where \(1\le j\le n\), \(1\le i\le n+1\) and~\(\epsilon\in\{0,1\}\),}
  \gamma_{i}\,s_{j} &= \begin{cases}
    s_{j+1}\,\gamma_{i} & \text{if \(i<j\),} \\
    s_{i+1}\,s_{i} & \text{if \(i=j\),} \\
    s_{j}\,\gamma_{i-1} & \text{if \(i>j\).}
  \end{cases}
\end{align}
where \(i\),~\(j\in\{1,\dots,n+1\}\).\footnote{%
  The identity for~\(d^{1}_{j+1}\gamma_{j}\) is stated incorrectly in~\cite[Def.~2]{BarceloEtAl:2021}.}
We say that a cube is \newterm{folded} if it lies in the image of some connection map.

All cubical sets we are going to consider will have connections.

\begin{example}
  \label{ex:singular-cubes}
  The singular \(n\)-cubes~\(x\colon I^{n}\to Z\) in a topological space~\(Z\) form the \(n\)-cubes of a cubical set (with connections).
  Here \(I=[0,1]\) is the unit interval; \(I^{0}\) is a point. The face, degeneracy and connection operators are\footnote{%
    Unlike~\cite[Sec.~2]{RiveraZeinalian:2018},~\cite[Def.~5.11]{MinichielloRiveraZeinalian:2023} and~\cite[Sec.~2.4]{MMRivera:2024},
    we use the minimum (and not the maximum) of adjacent coordinates for the connection operators.
    In general, both choices are possible, \cf~\cite[Rem.~11]{BarceloEtAl:2021}.
    The reason for our choice is that the minimum naturally appears in a crucial identity for the Szczarba operators,
    see \Cref{thm:properties-sz-new-s}\,\ref{qq2}.}
  \begin{align}
    (d^{\epsilon}_{i}x)(b_{1},\dots,b_{n-1}) &= x(b_{1},\dots,b_{i-1},\epsilon,b_{i},\dots,b_{n}), \\
    (s_{i}x)(b_{1},\dots,b_{n+1}) &= x(b_{1},\dots,b_{i-1},b_{i+1},\dots,b_{n+1}), \\
    (\gamma_{i}x)(b_{1},\dots,b_{n+1}) &= x(b_{1},\dots,b_{i-1},\min(b_{i},b_{i+1}),b_{i+2},\dots,b_{n+1}).
  \end{align}
\end{example}

The product of two cubical sets~\(Y\) and~\(Z\) is the cubical set~\(Y\times Z\) with
\begin{equation}
  \label{eq:def-product-cubical}
  (Y\times Z)_{n} = \Bigl(\, \bigsqcup_{k+l=n} Y_{k}\times Z_{l} \Bigr) \! \Bigm/ \sim
\end{equation}
for~\(n\ge0\), where \((s_{k+1}y,z)\sim(y,s_{1}z)\) for any~\(y\in Y_{k}\) and~\(z\in Z_{l}\),
\cf~\cite[p.~209]{KadeishviliSaneblidze:2005}.\footnote{%
  The identification of the last degeneracy map of~\(Y\) with the first one of~\(Z\),
  necessary to obtain a cubical set, is missing in~\cite[p.~6]{MMRivera:2024}.}
We write the equivalence class of~\((y,z)\) in~\(Y\times Z\) as~\([y,z]\); the structure maps are applied
componentwise to~\(y\) or~\(z\) according to their dimensions.

By a \newterm{cubical group} we mean a group object in the category of cubical sets (with connections).
\newterm{Cubical monoids} are defined analogously.

The normalized cubical chain complex~\(C(Y)\) with coefficients in a commutative ring~\(\kk\)
is obtained by dividing out the subcomplexes spanned by the degenerate simplices~\(s_{i}y\)
and folded simplices~\(\gamma_{i}x\) from the non-normalized complex.
The differential of an \(n\)-cube~\(y\) is
\begin{equation}
  d\,y = \sum_{i=1}^{n}(-1)^{i}\,(d^{0}_{i}x-d^{1}_{i}y).
\end{equation}
(The sign agrees with~Massey~\cite[Def.~II.2.3]{Massey:1980}, 
but it is different from the convention used by Rivera \emph{et al.}
\cite[Sec.~2, p.~3794]{RiveraZeinalian:2018},~\cite[p.~36]{MinichielloRiveraZeinalian:2023},~\cite[Sec.~2.4]{MMRivera:2024}.)
The complex~\(C(Y)\) is a dgc with the Serre diagonal
\begin{equation}
  \Delta\,y = \!\!\!\sum_{\substack{k+l=n \\ (\alpha,\beta)\in\Shuff(k,l)}}\!\!\!
  (-1)^{(\alpha,\beta)}\,
  d^{0}_{\beta_{1}}\cdots d^{0}_{\beta_{l}}\,y \otimes d^{1}_{\alpha_{1}}\cdots d^{1}_{\alpha_{k}}\,y
\end{equation}
and the augmentation sending every \(0\)-cube to~\(1\in\kk\).
For any two cubical sets~\(Y\) and~\(Z\) one has \(C(Y\times Z)=C(Y)\otimes C(Z)\) as dgcs, naturally in~\(Y\) and~\(Z\).

\section{Simplicial cubes}
\label{sec:simplicial-cubes}

Recall (from~\cite[Def.~5.4]{May:1968}, for instance) that the simplicial \(n\)-simplex~\(\Simp^{n}\), \(n\ge0\), is the simplicial set
whose \(m\)-simplices are the weakly increasing maps~\(\{0,\dots,m\}\to\{0,\dots,n\}\).
In particular, an \(m\)-simplex in the simplicial interval~\(\II=\Simp^{1}\) can be written as weakly increasing sequence
of \(m+1\) integers in the set~\(\TWO=\{0,1\}\). Hence, such a simplex is determined by the number~\(k\in\{1,\dots m+1\}\) of integers sent to~\(0\),
and we write it as~\([k]_{m}\). In this notation,
\begin{equation}
  d_{j}[k]_{m} = \begin{cases}
    [k]_{m-1} & \text{if \(k\le j\)} \\
    [k-1]_{m-1} & \text{if \(k>j\)}
  \end{cases},
  \qquad
  s_{j}[k]_{m} = \begin{cases}
    [k]_{m+1} & \text{if \(k\le j\)} \\
    [k+1]_{m+1} & \text{if \(k>j\)}
  \end{cases}
\end{equation}
for~\(0\le j\le m\) and~\(0\le k\le m+1\) (assuming \(m>0\) in the first formula).

Given an \(m\)-simplex~\(x=(x_{1},\dots,x_{n})\) in the simplicial cube~\(\II^{n}\),
we can list the component simplices as the rows of a matrix. For example, the matrix
\begin{equation}
  \label{eq:simplex-matrix}
  \begin{bmatrix}
    0 & 1 & 1 & 1 & 1 \\
    0 & 0 & 0 & 0 & 1 \\
    0 & 0 & 1 & 1 & 1
  \end{bmatrix}
\end{equation}
describes the \(4\)-simplex~\(x=([1]_{4},[4]_{4},[2]_{4})\) in~\(\II^{3}\), which we abbreviate to~\([1,4,2]_{4}\).
We call \eqref{eq:simplex-matrix} the \newterm{matrix form} of~\(x\).
Note that its columns are the vertices of~\(x\).
They form a weakly increasing sequence in~\(\TWO^{n}\) with respect to the componentwise partial order.
Conversely, any such sequence in~\(\TWO^{n}\) of length~\(m+1\)
determines an \(m\)-simplex in~\(\II^{n}\).

It will often be more convenient to write an \(m\)-simplex~\(x=[k_{1},\dots,k_{n}]_{m}\)
as the ordered partition~\(u=(u_{0},\dots,u_{m+1})\) of~\(\{1,\dots,n\}\)
where \(u_{j}=\{\,i\mid k_{i} = j\,\}\).
In this notation, the simplex~\(x=[1,4,2]_{4}\) from above is written as~\((\emptyset,1,3,\emptyset,2)\).
(We omit braces around singletons.) We call this the \newterm{partition form} of~\(x\).

Using this partition form, we have
\begin{align}
  \label{eq:d-j-u}
  d_{j}u &= \bigl(u_{0},\dots,u_{j}\cup u_{j+1},\dots, u_{m+1}\bigr), \\
  \label{eq:s-j-u}
s_{j}u &= \bigl(u_{0},\dots,u_{j},\emptyset,u_{j+1},\dots,u_{m+1})
\end{align}
for~\(0\le j\le m\) (and again \(m>0\) in the first formula).

From the formula~\eqref{eq:s-j-u} we immediately get the following.

\begin{lemma}
  An \(m\)-simplex~\(u\) in~\(\II^{n}\) (in partition form) is degenerate if and only if \(u_{j}=\emptyset\) for some~\(1\le j\le m\).
\end{lemma}

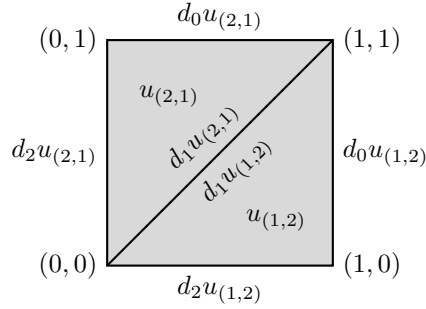
\begin{figure}
  \begin{tikzpicture}[scale=3]
    \filldraw[fill=black!15, thick] (0,0) node [below,left] {\((0,0)\)}
    -- node [below] {\(d_{2}u_{(1,2)}\)} (1,0) node [below,right] {\((1,0)\)}
    -- node [right] {\(d_{0}u_{(1,2)}\)} (1,1) node [above,right] {\((1,1)\)}
    -- node [above] {\(d_{0}u_{(2,1)}\)} (0,1) node [above,left] {\((0,1)\)}
    -- node [left] {\(d_{2}u_{(2,1)}\)} cycle;
    \draw [thick] (0,0) -- node [below,rotate=45] {\(d_{1}u_{(1,2)}\)} node [above,rotate=45] {\(d_{1}u_{(2,1)}\)}  (1,1);
    \draw (0.75,0.2) node {\(u_{(1,2)}\)};
    \draw (0.275,0.75) node {\(u_{(2,1)}\)};
  \end{tikzpicture}
  \caption{The simplicial square~\(\II^{2}\). Written in partition form,
  we have \(u_{(1,2)}=(\emptyset,1,2,\emptyset)\) and \(u_{(2,1)}=(\emptyset,2,1,\emptyset)\).}
  \label{fig:square}
\end{figure}

To a permutation~\(\pi\in S_{n}\) we associate the non-degenerate \(n\)-simplex
\begin{equation}
  u_{\pi} = \bigl(\emptyset,\pi(1),\dots,\pi(n),\emptyset\bigr) \in \II^{n}
\end{equation}
(again in partition form). For~\(n=2\) these simplices are shown in \Cref{fig:square}.

An \(m\)-face~\(v=(v_{0},\dots,v_{m+1})\) of~\(u_{\pi}\)
is determined by an increasing sequence \(0\le j_{1}<\dots<j_{n-m}\le n\) via
\begin{align}
  \label{eq:v-j-pi}
  v &= \partial_{j_{1}}\cdots\partial_{j_{n-m}} u_{\pi} \\
  \notag &= \Bigl(\bigl\{\pi(1),\dots,\pi(i_{0})\bigr\},\bigl\{\pi(i_{0}+1),\dots,\pi(i_{1})\bigr\},\dots \\
  \notag &\qquad\qquad\qquad \bigl\{\pi(i_{m-1}+1),\dots,\pi(i_{m})\bigl\},\bigl\{\pi(i_{m}+1),\dots,\pi(n)\bigl\}\Bigr),
\end{align}
where
\begin{equation}
  \{i_{0}<i_{1}<\dots<i_{m-1}<i_{m}\}=\{0,\dots,n\}\setminus\{j_{1},\dots,j_{n-m}\}.
\end{equation}
In other words, if we think of the parts of~\(u_{\pi}\) as separated by vertical bars,
\begin{equation}
  \label{eq:pi-bars}
  \emptyset \stackrel{0}{\mid}
  \pi(1) \stackrel{1}{\mid} \pi(2) \stackrel{2}{\mid} \pi(3) \stackrel{3}{\mid}
  \cdots \stackrel{\!\!n-2\!\!}{\mid} \pi(n-1) \stackrel{\!\!n-1\!\!}{\mid} \pi(n)
  \stackrel{n}{\mid} \emptyset,
\end{equation}
and we remove the bars labelled~\(j_{1}\),~\ldots,~\(j_{n-m}\),
then the result is the ordered partition defining \(v\).
This implies in particular that the \(j_{i}\)'s depend only on the sizes~\(n_{j}=\# v_{j}\) of the parts of the partition~\(v\).

\begin{lemma} \( \)
  \label{thm:cube-generated-embedding}
  \begin{enumroman}
  \item \label{thm:cube-generated-embedding-1}
    The simplicial cube~\(\II^{n}\) is generated by the simplices~\(u_{\pi}\), \(\pi\in S_{n}\).
    In other words, any simplex~\(u\in\II^{n}\) is a degeneration of a face of some~\(u_{\pi}\).
  \item \label{thm:cube-generated-embedding-2}
    Any~\(\pi\in S_{n}\) defines an embedding~\(\iota_{\pi}\colon\Simp^{n}\to\II^{n}\)
    that sends the unique non-degenerate \(n\)-simplex in~\(\Simp^{n}\) to~\(u_{\pi}\).
  \end{enumroman}
\end{lemma}

\begin{proof}
  Given any simplex~\(u\in\II^{n}\) (in partition form),
  choose an order of the elements of each set~\(u_{j}\), \(0\le j\le m+1\), and let \(\pi\) be the permutation
  given by the concatenation of these ordered sets. For example, if \(u=(\{1,3\},\emptyset,\{4\},\{2\})\),
  then we can choose \(\pi=(1,3,4,2)\) or~\(\pi=(3,1,4,2)\).
  It follows from the formulas~\eqref{eq:d-j-u} and~\eqref{eq:s-j-u} that \(u\) is a degeneration of a face of~\(u_{\pi}\).
  This proves the first claim.

  If an \(m\)-simplex~\(v\) lies in the image of a map~\(\iota_{\pi}\),
  then each of the sets~\(v_{0}\),~\dots,~\(v_{m+1}\) can be ordered in such a way that their concatenation gives \(\pi\).
  Now the sizes of the non-empty parts~\(v_{j}\) uniquely determine the face of~\(u_{\pi}\) of which \(v\) is a degeneration, and
  the other~\(v_{j}\) uniquely determine the degeneration operators. Hence \(\Simp^{n}\) embeds into~\(\II^{n}\) under~\(\iota_{\pi}\).
\end{proof}

We say that two distinct \(n\)-simplices in~\(\II^{n}\) are \newterm{adjacent} if they share a common facet.

\begin{lemma}
  \label{thm:simplices-adjacent}
  Let \(\pi\),~\(\pi'\in S_{n}\). Then the simplices~\(u_{\pi}\) and~\(u_{\pi'}\) are adjacent
  if and only if \(\pi'=\pi\tau\) for some adjacent transposition~\(\tau=(j,j+1)\), \(0<j<n\).
  In this case the \(j\)-th facets of~\(u_{\pi}\) and~\(u_{\pi'}\) agree: \(d_{j}u_{\pi}=d_{j}u_{\pi'}\).
\end{lemma}

\begin{proof}
  This follows directly from the formula~\eqref{eq:d-j-u}.
\end{proof}

\begin{lemma}
  \label{thm:cube-hereditary}
  Let \(\pi\ne\pi'\in S_{n}\), and let the simplex~\(v=u_{\pi}\cap u_{\pi'}\) be the intersection of~\(u_{\pi}\) and~\(u_{\pi'}\).
  Then there is a sequence \(\pi_{0}\),~\ldots,~\(\pi_{l}\) in~\(S_{n}\)  of length~\(l\ge 1\) such that
  \begin{enumarabic}
  \item \(\pi_{0}=\pi\) and \(\pi_{l}=\pi'\),
  \item \(u_{\pi_{j-1}}\) and~\(u_{\pi_{j}}\) are adjacent for any~\(1\le j\le l\),
  \item \(v\) is a face of~\(u_{\pi_{j}}\) for any~\(0\le j\le l\).
  \end{enumarabic}
\end{lemma}

In the language of Billera--Rose~\cite[p.~486]{BilleraRose:1992} this says
that the simplicial complex~\(\II^{n}\) is \newterm{hereditary}.

\begin{proof}
  Let \(m\) be the dimension of~\(v\).
  If \(w=(w_{0},\dots,w_{m+1})\) is a face of~\(u_{\pi}\) (in partition form) with \(w_{0}\ne\emptyset\), then
  \begin{equation}
    \tilde w = (\emptyset,w_{0},\dots,w_{m+1})
  \end{equation}
  is a face of~\(u_{\pi}\) having \(w\) as a facet. The same holds for~\(u_{\pi'}\).
  Since \(v\) is the largest common face of~\(u_{\pi}\) and~\(u_{\pi'}\),
  we must therefore have \(v_{0}=\emptyset\) and analogously \(v_{m+1}=\emptyset\).

  Note that if we write
  \begin{equation}
    v = \partial_{j_{1}}\cdots\partial_{j_{n-m}} u_{\pi}
  \end{equation}
  with~\(j_{1}<\dots<j_{n-m}\), then it follows from the discussion around~\eqref{eq:pi-bars} that we also have
  \begin{equation}
    v = \partial_{j_{1}}\cdots\partial_{j_{n-m}} u_{\pi'}
  \end{equation}
  for the same~\(j_{i}\)'s. Thus, in the notation from~\eqref{eq:v-j-pi} we must have \(i_{0}=0\), \(i_{m}=n\) and
  \begin{equation}
    \begin{split}
      \bigl\{\pi(1),\dots,\pi(i_{1})\bigr\} &= \bigl\{\pi'(1),\dots,\pi'(i_{1})\bigr\}, \\
      \bigl\{\pi(i_{1}+1),\dots,\pi(i_{2})\bigr\} &= \bigl\{\pi'(i_{1}+1),\dots,\pi'(i_{2})\bigr\}, \\
      &{\phantom{x}\vdots} \\
      \bigl\{\pi(i_{m-1}+1),\dots,\pi(n)\bigr\} &= \bigl\{\pi'(i_{m-1}+1),\dots,\pi'(n)\bigr\}.
    \end{split}
  \end{equation}

  This implies
  \begin{equation}
    \label{eq:pi-tau}
    \pi' = \pi\circ(\tau_{1}\sqcup\dots\sqcup\tau_{m})
  \end{equation}
  for some permutations~\(\tau_{j}\in S_{n_{j}}\) where \(n_{j}=\# v_{j}\) for~\(1\le j\le m\).
  Note that for any choice of permutations~\(\tau_{j}\in S_{n_{j}}\) the right-hand side of~\eqref{eq:pi-tau}
  defines a permutation~\(\tilde\pi\in S_{n}\) whose associated simplex~\(u_{\tilde\pi}\) has \(v\) as a face.
  Since each~\(\tau_{j}\) is a product of adjacent transpositions, the claim follows.
\end{proof}

\begin{example}
  We illustrate the formula~\eqref{eq:pi-tau} from the preceding proof with
  \begin{equation}
    \pi = \begin{pmatrix}
      1 & 2 & 3 & 4 & 5 & 6 \\
      5 & 4 & 2 & 1 & 3 & 6
    \end{pmatrix}
    \quad\text{and}\quad
    \pi' = \begin{pmatrix}
      1 & 2 & 3 & 4 & 5 & 6 \\
      2 & 4 & 5 & 1 & 6 & 3
    \end{pmatrix} \in S_{6}.
  \end{equation}
  In this case the intersection of~\(u_{\pi}\) and~\(u_{\pi'}\) is the \(3\)-simplex (in partition form)
  \begin{equation}
    v = (\emptyset,\{2,4,5\},\{1\},\{3,6\},\emptyset),
  \end{equation}
  and we have
  \begin{equation}
    \tau_{1} = \begin{pmatrix}
      1 & 2 & 3 \\ 3 & 2 & 1
    \end{pmatrix},
    \qquad
    \tau_{2} = \begin{pmatrix}
      1 \\ 1
    \end{pmatrix},
    \qquad
    \tau_{3} = \begin{pmatrix}
      1 & 2 \\ 2 & 1
    \end{pmatrix}.
  \end{equation}
\end{example}

\begin{corollary}
  \label{thm:triang-map}
  Let \(X\) be a simplicial set and \(n\ge0\).
  Given a family~\((x_{\pi})_{\pi\in S_{n}}\) of \(n\)-simplices in~\(X\),
  the assignment \(u_{\pi} \mapsto f(u_{\pi}) = x_{\pi}\) for all~\(\pi\in S_{n}\) extends
  to a simplicial map~\(f\colon \II^{n}\to X\) if and only if
  \begin{equation*}
    d_{j}\,x_{\pi} = d_{j}\,x_{\pi\tau}
  \end{equation*}
  for any~\(\pi\in S_{n}\) and any~\(0<j<n\), where \(\tau\) is the adjacent transposition~\((j,j+1)\).
  Moreover, such an extension is unique.
\end{corollary}

\begin{proof}
  Uniqueness follows from the fact that the simplices~\(u_{\pi}\) generate \(\II^{n}\)
  as a simplicial set, see \Cref{thm:cube-generated-embedding}\,\ref{thm:cube-generated-embedding-2}.

  By part~\ref{thm:cube-generated-embedding-1} of the same result,
  each \(u_{\pi}\) generates a simplicial subset of~\(\II^{n}\) isomorphic to the \(n\)-simplex~\(\Simp^{n}\).
  Hence the assignment~\(u_{\pi}\mapsto x_{\pi}\) uniquely extends to each such embedded \(n\)-simplex.
  Call this extension~\(f_{\pi}\). We have to show that these extensions are compatible with each other
  if and only if the stated condition holds.

  Consider two simplices~\(u_{\pi}\) and~\(u_{\pi'}\) that share a common facet~\(v\).
  By \Cref{thm:simplices-adjacent}, we have \(\pi'=\pi\,\tau\) for some adjacent transposition~\(\tau=(j,j+1)\)
  and also \(v=d_{j}\,u_{\pi}=d_{j}\,u_{\pi'}\). Hence the two maps~\(f_{\pi}\) and~\(f_{\pi'}\) agree on their common facet
  if and only if \(d_{j}\,x_{\pi} = d_{j}\,x_{\pi'}\). This proves the necessity of the condition.

  For sufficiency, consider arbitrary~\(\pi\),~\(\pi'\in S_{n}\) and set \(v=u_{\pi}\cap u_{\pi'}\).
  According to \Cref{thm:cube-hereditary}, we can connect \(u_{\pi}\) and~\(u_{\pi'}\) by a sequence of simplices
  \begin{equation}
    u_{\pi}=u_{\pi_{0}}, u_{\pi_{1}}, \dots, u_{\pi_{l}} = u_{\pi'}
  \end{equation}
  such that for any~\(1\le j\le l\), the simplices~\(u_{\pi_{j}}\) and~\(u_{\pi_{j+1}}\) share a common facet containing \(v\).
  Then our assumption implies that all~\(f_{\pi_{j}}\) agree on~\(v\), in particular \(f_{\pi}\) and~\(f_{\pi'}\) do so.
  Hence we get a well-defined simplicial map from~\(\II^{n}\) to~\(X\).
\end{proof}

We finally look more closely at simplices on products of simplicial cubes.
The notation~``\(s_{\beta-1}u_{\sigma}\)'' is as in~\eqref{eq:def-s-alpha-1}.

\begin{proposition}
  \label{thm:decomposition-simplex-product-cube}
  Let \(k\),~\(l\ge0\). Assume that
  the permutation~\(\pi\in S_{k+l}\) corresponds to the \((k,l)\)-shuffle~\((\alpha,\beta)\)
  and the permutations~\(\sigma\in S_{k}\),~\(\tau\in S_{l}\)
  under the bijection~\(\Psi_{k,l}\) from \Cref{thm:Psi-k-l}.
  The \((k+l)\)-simplex~\(u_{\pi}\in\II^{k+l}\) then corresponds to the pair~\((v,w)\)
  under the decomposition~\(\II^{k+l}=\II^{k}\times\II^{l}\), where
  \begin{equation*}
    v = s_{\beta-1}\,u_{\sigma} \in \II^{k}
    \qquad\text{and}\qquad
    w = s_{\alpha-1}\,u_{\tau} \in \II^{l}.
  \end{equation*}
\end{proposition}

\begin{proof}
  The simplex~\(v\) is the image of~\(u\) under the projection~\(\II^{k+l}\to\II^{k}\).
  In the partition form we have
  \(v_{j}=\{i \mid i\in u_{j}, i\le k\}\) for any~\(0\le j\le k+l+1\).
  In other words, all occurrences of numbers~\(>k\) are taken out. Since \(\beta\) records the positions
  of these elements and \(\sigma\) is the sequence of the elements in~\(\pi\) that are~\(\le k\),
  we conclude from~\eqref{eq:s-j-u} that \(v = s_{\beta-1}\,u_{\sigma}\) holds.
  The second identity is analogous; this time we have
  \(w_{j}=\{i-k \mid i\in u_{j}, i\ge k+1\}\).
\end{proof}

\section{The cube category}

Recall that we write \(\TWO=\{0,1\}\).
The cube category (with connections)~\(\Cube\) has the sets~\(\TWO^{n}\) with~\(n\ge0\) as objects.
The morphisms are generated by the maps
\begin{align}
  \delta^{\epsilon}_{i}\colon \TWO^{n} &\to \TWO^{n+1}, & t &\mapsto (t_{1},\dots,t_{i-1},\epsilon,t_{i},\dots,t_{n}),
  \quad 0\le i\le n,\;\epsilon\in\{0,1\} \\
  \label{eq:cubecat:s}
  \sigma_{i}\colon \TWO^{n} &\to \TWO^{n-1}, & t &\mapsto (t_{1},\dots,t_{i-1},t_{i+1},\dots,t_{n}),
  \quad 1\le i\le n, \\
  \gamma_{i}\colon \TWO^{n} &\to \TWO^{n-1}, & t &\mapsto (t_{1},\dots,t_{i-1},\min(t_{i},t_{i+1}),\dots,t_{n}),
  \quad 1\le i<n.
\end{align}
Cartesian products of sets give a monoidal structure to the cube category.
The morphisms are monoidally generated by~\(\delta^{\epsilon}_{1}\colon\TWO^{0}\to\TWO^{1}\) with~\(\epsilon\in\{0,1\}\),
 \(\sigma_{1}\colon\TWO^{1}\to\TWO^{0}\) and~\(\gamma_{1}\colon\TWO^{2}\to\TWO^{1}\), compare \cite[Sec.~2.4]{MMRivera:2024}.

The assignment
\begin{equation}
  \TWO^{n} \mapsto \II^{n}
\end{equation}
becomes a monoidal functor~\(\Cube\to\sSet\) if we map the morphisms as follows:
The morphism~\(\delta^{\epsilon}_{1}\colon\TWO^{0}\to\TWO^{1}\) is sent to the inclusion of~\(\II^{0}=\Simp^{0}\) as the vertex~\(\epsilon\in\{0,1\}\) of~\(\II\),
and \(\sigma_{1}\colon\TWO^{1}\to\TWO^{0}\) to the projection~\(\II\to\Simp^{0}\).
The simplicial morphism associated to~\(\gamma_{1}\colon\TWO^{2}\to\TWO^{1}\) sends
the two non-degenerate \(2\)-simplices~\(u_{(1,2)}\) and~\(u_{(2,1)}\) to~\(s_{0}u_{(1)}\).
This is well-defined by \Cref{thm:triang-map} because the diagonal~\(d_{1}u_{(1,2)}=d_{1}u_{(2,1)}\) of the square
is send to~\(u_{(1)}\) for both simplices, compare \Cref{fig:square}.

Recall that the \(m\)-simplices in the simplicial \(n\)-cube~\(\II^{n}\) are in bijection
with the weakly increasing sequences of length~\(m+1\) in the set~\(\TWO^{n}\) representing the vertices of such a simplex.
One can verify from the description of the generating morphisms above
that any morphism~\(\lambda\colon \TWO^{n}\to \TWO^{n'}\) in the cube category preserves
weakly increasing sequences, so that it defines a simplicial map~\(\lambda_{*}\colon\II^{n}\to\II^{n'}\) between simplicial cubes.
The assignment~\(\lambda\mapsto\lambda_{*}\) is functorial.

Let \(n\ge1\).
We say that a permutation~\(\tilde\pi\in S_{n-1}\) is obtained from~\(\pi\in S_{n}\) ``by removing the assignment \(i\mapsto\pi(i)\)''
if the remaining elements of~\(\{1,\dots,n\}\) are renumbered and the mappings between them preserved.
We may also describe this by saying that \(\pi\) is obtained from~\(\tilde\pi\) ``by adding the assignment \(i\mapsto\pi(i)\)''.
In formulas,
\begin{equation}
  \tilde\pi(j) = \begin{cases}
    \pi(j) & \text{if \(j<i\) and \(\pi(j)<\pi(i)\),} \\
    \pi(j)-1 & \text{if \(j<i\) and \(\pi(j)>\pi(i)\),} \\
    \pi(j-1) & \text{if \(j>i\) and \(\pi(j)<\pi(i)\),} \\
    \pi(j-1)-1 & \text{if \(j>i\) and \(\pi(j)>\pi(i)\).}
  \end{cases}
\end{equation}
For example, if
\begin{equation}
  \pi = \begin{pmatrix}1 & 2 & 3 & 4 \\ 3 & 4 & 2 & 1\end{pmatrix}\in S_{4}
  \qquad\text{and}\qquad
  \tilde\pi = \begin{pmatrix}1 & 2 & 3 \\ 2 & 3 & 1\end{pmatrix}\in S_{3},
\end{equation}
then \(\tilde\pi\) is obtained from~\(\pi\) by removing the assignment~\(3\mapsto2\).

\begin{lemma}
  \label{thm:d-u-tilde-pi}
  Let \(\pi\in S_{n-1}\) and \(1\le i\le n\).
  \begin{enumroman}
  \item If \(\tilde\pi\in S_{n}\) be obtained from~\(\pi\) by adding the assignment~\(1\mapsto i\), then
    \begin{equation*}
      d_{0}\,u_{\tilde\pi} = (\delta^{1}_{i})_{*}\,u_{\pi}.
    \end{equation*}
  \item If \(\tilde\pi\in S_{n}\) be obtained from~\(\pi\) by adding the assignment~\(n\mapsto i\), then
    \begin{equation*}
      d_{n}\,u_{\tilde\pi} = (\delta^{0}_{i})_{*}\,u_{\pi}.
    \end{equation*}
  \end{enumroman}
\end{lemma}

\begin{proof}
  We only discuss the first claim; the second one is analogous.

  If \(\tilde\pi\in S_{n}\) be obtained from~\(\pi\) by adding the assignment~\(1\mapsto i\), then
  \(d_{0}\,u_{\tilde\pi}=(\tilde\pi(1),\dots,\tilde\pi(n),\emptyset)\) (in partition form)
  differs from~\(u_{\pi}=(\emptyset,\pi(1),\dots,\pi(n-1),\emptyset)\)
  by the entry~\(\tilde\pi(1)=i\) in the first position and by increasing all other entries that are~\(\ge i\).
  In the matrix representation~\eqref{eq:simplex-matrix}, this corresponds to adding a row consisting entirely of~\(1\)'s
  before the \(i\)-th row. This means that \(u_{\tilde\pi}\in\II^{n}\) is obtained from~\(u_{\pi}\in\II^{n-1}\)
  by adding an \((n-1)\)-simplex that lies over the end point~\(1\) of the simplicial interval in the \(i\)-th coordinate.
  This is exactly what the map~\((\delta^{1}_{i})_{*}\) does.
\end{proof}

\begin{lemma}
  \label{thm:s-u-tilde-pi}
  Let \(\pi\in S_{n+1}\) and \(1\le i\le n+1\).
  \begin{enumroman}
  \item \label{thm:s-u-tilde-pi-1}
    If \(\tilde\pi\) be obtained from~\(\pi\) by removing the assignment~\(j\mapsto\pi(j)=i\), then
    \begin{equation*}
      s_{j-1}\,u_{\tilde\pi} = (\sigma_{i})_{*}u_{\pi}.
    \end{equation*}
  \item \label{thm:s-u-tilde-pi-2}
  Assume \(i\le n\), and set \(j=\min(\pi^{-1}(i),\pi^{-1}(i+1))\).
    If \(\tilde\pi\) is obtained from~\(\pi\) by removing the assignment~\(j\mapsto\pi(j)\), then
    \begin{equation*}
      s_{j-1}\,u_{\tilde\pi} = (\gamma_{i})_{*}u_{\pi}.
    \end{equation*}
  \end{enumroman}
\end{lemma}

\begin{proof}
    Assume that \(\tilde\pi\) be obtained from~\(\pi\) by removing the assignment~\(j\mapsto\pi(j)=i\).
    Then \(s_{j-1}\,u_{\tilde\pi}\) differs from~\(u_{\pi}\) by replacing the value~\(i\) by the empty set
    and decreasing all larger values by~\(1\). In the matrix representation~\eqref{eq:simplex-matrix},
    this corresponds to deleting the \(i\)-th row. in other words, the \(i\)-th coordinate of~\(u_{\pi}\) is removed.
    This is the effect of the simplicial map~\((\sigma_{i})_{*}\).

    We now consider the second claim. We write \(\pi(j)=i'\in\{i,i+1\}\) for the value that precedes the other one
    in the sequence~\((\pi(1)),\dots,\pi(n+1))\). Then \(s_{j-1}\,u_{\tilde\pi}\) differs from~\(u_{\pi}\)
    by replacing \(i'\) by the empty set and decreasing all larger values by~\(1\). In the matrix representation~\eqref{eq:simplex-matrix},
    this corresponds to deleting either the row at position~\(i\) or the following one,
    whichever contains fewer \(0\)'s. Hence the remaining row is obtained from these two adjacent rows
    by taking the minimum in each coordinate. That is exactly what the map~\((\gamma_{i})_{*}\) does.
\end{proof}

The standard \(n\)-cube~\(\Cube^{n}\) is the cubical set with \(k\)-cubes~\(\Cube^{n}_{k}=\Hom(\TWO^{k},\TWO^{n})\)
and the structure maps induced from the morphisms in the cube category. Then \(\Cube^{n}=(\Cube^{1})^{n}\) as cubical sets.
Moreover, the dgc isomorphism~\(C(\Cube^{1})=C(\II)\) together with the shuffle map~\eqref{eq:shuffle-map}
induces a dgc isomorphism~\(C(\Cube^{n})=C(\II)^{\otimes n}\).

\section{Triangulating cubical sets}
\label{sec:triangulate-cubical-sets}

Given an \(n\)-cube~\(y\) in a cubical set~\(Y\), we use the abbreviations
\begin{equation}
  \TWO_{y}=\TWO^{n},
  \qquad
  \Cube_{y}=\Cube^{n}
  \qquad
  \II_{y}=\II^{n}
  \qquad\text{and}\qquad
  I_{y}=I^{n},
\end{equation}
where the latter three symbols denote the cubical, simplicial and topological \(n\)-cube, respectively.

Since any cubical set~\(Y\) is the colimit over its cubes, we get a triangulation functor~\(\TT\)
from the category of cubical sets to that of simplicial sets as follows:
We consider the disjoint union of simplicial cubes that has one simplicial cube~\(\II_{y}\)
for each cube~\(y\in Y\). We denote a simplex~\(u\in\II_{y}\) also by~\((y, u)\).
In the simplicial set~\(\TT Y\) associated to a cubical set~\(Y\),
the simplex~\((\lambda^{*}y,u)\) is identified with the simplex~\((y,\lambda_{*}u)\)
for any~\(y\in Y\) and any morphism~\(\lambda\colon \TWO_{\lambda^{*}y}\to\TWO_{y}\) in the cube category,
\begin{equation}
  \TT Y = \bigsqcup_{y\in Y} \II_{y} \,\Bigm/ (\lambda^{*}y,u) \sim (y,\lambda_{*}u).
\end{equation}
We write \([y,u]\) for the equivalence class of~\(u\in\II_{y}\) in~\(\TT Y\).
Note that the definition of~\(\TT Y\) is analogous to that of the geometric realization of~\(Y\),
\begin{equation}
  |Y| = \bigsqcup_{y\in Y} I_{y} \,\Bigm/ (\lambda^{*}y,u) \sim (y,\lambda_{*}u).
\end{equation}
Here \(\lambda_{*}\colon I_{\lambda^{*}y}\to I_{y}\)
denotes the geometric realization of the morphism~\(\lambda\) (which have appeared in \Cref{ex:singular-cubes} already).
We refer to~\cite[\S 14]{May:1968} for the similarly defined geometric realization~\(|X|\) of a simplicial set~\(X\).

\begin{lemma}
  \label{thm:homeo-TY-Y}
  For any cubical set~\(Y\) there is a homeomorphism
  \begin{equation*}
    |\TT Y| \approx |Y|,
  \end{equation*}
  natural in~\(Y\).
\end{lemma}

See Carranza--Kapulkin~\cite[Lemma~2.23]{CarranzaKapulkin:2023} for the analogous result
in the context of cubical sets with two kinds of connection operators.

\begin{proof}
  The homeomorphism~\(|\II|\approx\TopI\) induces a homeomorphism
  \begin{equation}
    \label{eq:homeo-IIn}
    |\II^{n}| \approx |\II|^{n} \approx \TopI^{n}
  \end{equation}
  since \(\II\) is a countable simplicial set, \cf~\cite[Thm.~14.3, Rem.~14.4]{May:1968}.

  One can check that this homeomorphism is compatible with cubical morphisms.
  For the maps~\(\delta_{i}^{\epsilon}\) and~\(\sigma_{i}\) this follows from the \(1\)-dimensional case.
  For the connections~\(\gamma_{i}\) it is useful to have the explicit maps
  \begin{align}
    |\II^{n}| &\to \TopI^{n},
    \quad
    [u,a] \mapsto b=(b_{1},\dots,b_{n})
  \end{align}
  for an \(m\)-simplex~\(u\) (in partition form) and \(a=(a_{0},\dots,a_{m})\in\Delta^{m}\subset\R^{m+1}\), where
  \begin{equation}
    b_{j} = a_{i}+\dots+a_{m},
  \end{equation}
  and \(i\in\{0,\dots,m+1\}\) is determined by the condition~\(j\in u_{i}\).
  (Recall from \Cref{ex:singular-simplex} that the coordinates of a point~\(a\in\Delta^{m}\) sum up to~\(1\).)
  Equivalently, the \(m\)-simplex in the \(j\)-th component of~\(u\) starts with \(i\)~zeroes.
  The inverse map is as follows: Given \(b\in\TopI^{n}\), choose a permutation~\(\pi\in S_{n}\) such that
  \begin{equation}
    b_{\pi(1)} \ge \dots \ge b_{\pi(n)}.
  \end{equation}
  Then \(b\) is sent to~\([u_{\pi},a]\), where
  \begin{equation}
    a_{0} = 1 - b_{\pi(1)},
    \quad
    a_{i} = b_{\pi(i)} - b_{\pi(i+1)}
    \quad\text{for~\(0<i<n\),}\quad
    a_{n} = b_{\pi(n)}.
  \end{equation}

  As a consequence, the homeomorphism~\eqref{eq:homeo-IIn} descends to the geometric realizations~\(|\TT Y|\) and~\(|Y|\).
\end{proof}

\begin{proposition}
  Let \(Y\) be a cubical set. The map
  \begin{equation*}
    \tt=\tt_{Y}\colon C(Y) \mapsto C(\TT Y),
    \qquad
    y \mapsto \sum_{\pi\in S_{n}} (-1)^{\deg\pi}\,(y,u_{\pi})
  \end{equation*}
  where \(n=\deg y\), is a quasi-isomorphism of dgcs, natural in~\(Y\).
\end{proposition}

\begin{proof}
  Recall that \(C(\Cube^{n}) = C(\II)^{\otimes n}\) as dgcs.
  We can therefore consider the shuffle map as a dgc map
  \begin{equation}
    \shuffle\colon C(\Cube^{n}) = C(\II)^{\otimes n} \to C\bigl(\II^{n}\bigr),
  \end{equation}
  which descends to the stated dgc morphism~\(\tt_{Y}\).

  Let us now verify that \(\tt\) is a quasi-isomorphism. We will use that
  \(|Y|\) is a CW~complex with one \(n\)-cell for each cube~\(y\in Y_{n}\) that is neither degenerate nor folded,
  see \cite[Prop.~II.1.8~\&~p.~58]{Baues:1980}.
  Similarly, \(|\TT Y|\) is a CW~complex with one \(n\)-cell for each non-degenerate \(n\)-simplex, \cf~\cite[Thm.~14.1]{May:1968}.
  It follows from the description of the homeomorphism~\(|\TT Y|\to|Y|\) in the proof of \Cref{thm:homeo-TY-Y}
  that the cell given by~\(y\) as above corresponds to the union of the cells~\([y,u]\)
  where \(u\in\II_{y}\) is non-degenerate and not contained in the boundary of~\(\II_{y}\).

  Since the shuffle map induces an isomorphism
  \begin{equation}
    H_{*}(\Cube^{n},\partial\Cube^{n}) \to H_{*}(\II^{n},\partial\II^{n})
  \end{equation}
  for any~\(n\ge0\), we also get an isomorphism
  \begin{equation}
    H_{*}\bigl(|Y|_{n},|Y|_{n-1}\bigr) \to H_{*}\bigl(|\TT Y|_{n},|\TT Y|_{n-1}\bigr)
  \end{equation}
  where \(|Y|_{n}\subset|Y|\) is the \(n\)-skeleton of~\(|Y|\) and \(|\TT Y|_{n}\) its image in~\(|\TT Y|\).
  It follows inductively that the shuffle map is a quasi-isomorphism for all skeleta of~\(|Y|\), hence for~\(|Y|\) itself.
\end{proof}

\begin{proposition}
  \label{thm:triangulation-product}
  Let \(Y\) and~\(Z\) be cubical sets. There is an isomorphism of simplicial sets
  \begin{equation*}
    \ff=\ff_{Y,Z}\colon \TT Y\times\TT Z \to \TT(Y\times Z)
  \end{equation*}
  that is natural in~\(Y\) and~\(Z\), associative in the obvious way and compatible with diagonals in the sense that
  \begin{equation*}
    \TT\Delta_{Y} = \ff\circ\Delta_{\TT Y}.
  \end{equation*}
  Moreover, the following diagram commutes:
  \begin{equation*}
    \begin{tikzcd}
      C(Y) \otimes C(Z) \arrow{d}[left]{\tt_{Y}\otimes\tt_{Z}} \arrow{rr}{\cong} & & C(Y\times Z) \arrow{d}{\tt_{Y\times Z}} \\
      C(\TT Y) \otimes C(\TT Z) \arrow{r}{\shuffle} & C(\TT Z\times\TT Z) \arrow{r}[above]{C(\ff)}[below]{\cong} & C\bigl(\TT(Y\times Z)\bigr)
    \end{tikzcd}
  \end{equation*}
\end{proposition}

The analogous result for cubical sets with two kinds of connection operators
has appeared in Doherty \emph{et al.}~\cite[Prop.~1.29]{DohertyEtAl:2024}.

\begin{proof}
  Let \(y\in Y_{k}\) and~\(z\in Z_{l}\). The map~\(\ff\) sends the \(m\)-simplex
  \begin{equation*}
    \bigl([y,v],[z,w]\bigr)\in\TT Y\times\TT Z
  \end{equation*}
  to~\(\bigl[[y,z],(v,w)\bigr]\in\TT(Y\times Z)\).
  Here \((v,w)\in\II_{[y,z]}\) is the product of the simplices \(v\in\II_{y}\) and~\(w\in\II_{z}\).

  For the other direction, let \(y\in Y\) be a \(k\)-cube and \(z\in Z\) an \(l\)-cube.
  Consider the mapping
  \begin{equation}
    \label{eq:triang-product-iso-inv}
    \bigl((y,z), u) \mapsto \Bigl( \bigl[y,(\pi_{1,k})_{*}u\bigr], \bigl[z,(\pi_{k+1,k+l})_{*}u\bigr] \Bigr)
  \end{equation}
  where \(\pi_{1,k}\colon\TWO^{k+l}\to\TWO^{k}\) is the projection onto the first~\(k\) 
  and \(\pi_{k+1,k+l}\colon\TWO^{k+l}\to\TWO^{l}\) that onto the last~\(l\) coordinates.
  The map~\eqref{eq:triang-product-iso-inv} is compatible with the identification~\((s_{k+1}\,y,z)\sim(y,s_{1}\,z)\) made for~\(Y\times Z\):
  Indeed, we have
  \begin{align}
    \bigl((s_{k+1}\,y,z), u) \mapsto\; & \Bigl( \bigl[s_{k+1}\,y,(\pi_{1,k+1})_{*}u\bigr], \bigl[z,(\pi_{k+2,k+l})_{*}u\bigr] \Bigr) \\
    \notag &= \Bigl( \bigl[y,(\sigma_{k+1}\pi_{1,k+1})_{*}u\bigr], \bigl[z,(\pi_{k+2,k+l})_{*}u\bigr] \Bigr) \\
    \notag &= \Bigl( \bigl[y,(\pi_{1,k})_{*}u\bigr], \bigl[z,(\pi_{k+2,k+l})_{*}u\bigr] \Bigr),
  \end{align}
  and also
  \begin{align}
    \bigl((y,s_{1}\,z), u) \mapsto\; & \Bigl( \bigl[y,(\pi_{1,k})_{*}u\bigr], \bigl[s_{1}\,z,(\pi_{k+1,k+l})_{*}u\bigr] \Bigr) \\
    \notag &= \Bigl( \bigl[y,(\pi_{1,k})_{*}u\bigr], \bigl[z,(\sigma_{1}\pi_{k+1,k+l})_{*}u\bigr] \Bigr) \\
    \notag &= \Bigl( \bigl[y,(\pi_{1,k})_{*}u\bigr], \bigl[z,(\pi_{k+2,k+l})_{*}u\bigr] \Bigr).
  \end{align}
  One can also check that \eqref{eq:triang-product-iso-inv} is compatible with the identifications coming
  from the definition of the triangulation of a cubical set and that the resulting bijection~\(\ff\) is simplicial.
  One can likewise check that \(\ff\) is associative and compatible with the diagonals.

  The diagram displayed above is induced by the diagram
  \begin{equation}
    \begin{tikzcd}
      C(\Cube_{y})\otimes C(\Cube_{z}) \arrow{d}[left]{\tt_{\Cube_{y}}\otimes\,\tt_{\Cube_{z}}} \arrow{r}{\cong} & C(\Cube_{[y,z]}) \arrow{d}{\tt_{\Cube_{[y,z]}}} \\
      C(\II_{y})\otimes C(\II_{z}) \arrow{r}{\shuffle} & C(\II_{[y,z]})
    \end{tikzcd}
  \end{equation}
  for all~\(y\in Y\) and~\(z\in Z\). The associativity of the shuffle map ensures that the latter commutes, hence also the former.
\end{proof}

\begin{corollary}
  \label{thm:triangulation-product-group}
  Let \(Y\) be a cubical group (or monoid). Then \(\TT Y\) is canonically a simplicial group (or monoid).
  Moreover, the dgc quasi-isomorphism
  \begin{equation*}
    \tt_{Y}\colon C(Y) \to C(\TT Y)
  \end{equation*}
  is multiplicative and therefore a quasi-isomorphism of dg~bialgebras.
\end{corollary}

\begin{proof}
  The product in~\(\TT Y\) is defined as the composition
  \begin{equation}
    \TT Y \times \TT Y \xrightarrow{\ff_{Y,Y}} \TT(Y\times Y) \xrightarrow{\TT\mu} \TT Y,
  \end{equation}
  where \(\mu\) is the product in~\(Y\).
  This results in a simplicial monoid because the isomorphism~\(\ff\) is associative.
  If \(Y\) is a cubical group with inversion map~\(\iota\colon Y\to Y\), then the compatibility of~\(\ff\) with the diagonals
  ensures that \(\TT Y\) becomes a simplicial group with inversion map~\(\TT\iota\).
  The commutative diagram from \Cref{thm:triangulation-product} finally implies that the map~\(\tt_{Y}\) is multiplicative.
\end{proof}

\begin{proposition}
  \label{thm:triangulation-map}
  Let \(Y\) be a cubical set, and let \(X\) be a simplicial set.
  Assume that for every~\(y\in Y\) we have a simplicial map
  \begin{equation*}
    F_{y}\colon \II_{y} \to X.
  \end{equation*}
  These maps induce a simplicial map
  \begin{equation*}
    F\colon \TT Y \to X,
    \qquad
    [y,u] \mapsto F_{y}(u)
  \end{equation*}
  if and only if
  \begin{equation*}
    F_{\lambda^{*}y}(u)=F_{y}(\lambda_{*}u)
  \end{equation*}
  for any \(y\in Y\), any morphism~\(\lambda\colon\TWO_{\lambda^{*}y}\to\TWO_{y}\) in the cube category
  and any simplex~\(u\in\II_{\lambda^{*}y}\).
\end{proposition}

\begin{proof}
  This follows immediately from the identifications made in definition of the triangulation~\(\TT Y\).
\end{proof}

\section{The Szczarba operators}

We write \([n]=\{0,\dots,n\}\) and
\begin{equation}
  \SS_{n}=[n-1]\times[n-2]\times\dots\times[0]
\end{equation}
for~\(n\ge0\); we also write \(\emptyset\) for the unique element of~\(\SS_{0}\).
The degree~\(\deg{\ii}\) of~\(\ii\in\SS_{n}\) is the sum of its components.

Let \(X\) be a simplicial set, \(G\) a simplicial group and \(\tau\colon X_{>0}\to G\) a twisting function.
We refer to~\cite[eq.~(1.1)]{Szczarba:1961},~\cite[Def.~18.3]{May:1968},~\cite[Sec.~1.3]{HessTonks:2006}
or~\cite[Sec.~2.5]{Franz:szczarba2} for the definition of a twisting function.
Based on~\(\tau\) one defines the Szczarba operators
\begin{equation}
  \Sz_{\ii}\colon X_{n}\to G_{n-1}
\end{equation}
of degree~\(-1\) for~\(n\ge1\) and~\(\ii\in\SS_{n-1}\), see~\cite[Def.~5]{HessTonks:2006} or~\cite[eq.~(5.2)]{Franz:szczarba2};
they of course go back to~\cite{Szczarba:1961}.
In terms of these, Szczarba's twisting cochain is given by
\begin{equation}
  \label{eq:szczarba-twisting-cochain}
  \tSz\colon C(X) \to C(G),
  \qquad
  x \mapsto \begin{cases}
    0 & \text{if \(n=0\),} \\
    \Sz_{\emptyset}x - 1 & \text{if \(n=1\),} \\
    \sum_{\ii\in\SS_{n-1}} (-1)^{\deg{\ii}}\,\Sz_{\ii}x & \text{if \(n\ge2\)} \\
  \end{cases}
\end{equation}
where \(n=\deg{x}\).

\begin{remark}
  Like the Szczarba operators, the twisting cochain~\(\tSz\) involves the inversion map on the simplicial group~\(G\).
  Its value on a $2$-simplex~$x\in X$ for instance is
  \begin{equation}
    \tSz(x) = \tau(x)^{-1}\tau(s_{1}d_{0}x)^{-1},
  \end{equation}
  \cf~\cite[eq.~(5.5)]{Franz:szczarba2}.
  Hence \(\tSz\) is not defined for a simplicial monoid~\(M\) instead of a simplicial group~\(G\)
  although a twisting function can be defined in that context. As mentioned in the introduction,
  Cai~\cite{Cai:twisting} considers a modified twisting cochain that does not involve the group inversion.
\end{remark}

Instead of repeating the precise definition of~\(\Sz_{\ii}\), we only state how these operators
interact with face and degeneracy maps.
For this we need to recall Szczarba's definition~\cite[Lemma~3.3]{Szczarba:1961} of a bijection
\begin{equation}
  \label{eq:def-Xi}
  \Xi\colon \SS_{n} \to \bigsqcup_{k+l=n-1} \Shuff(k,l)\times\SS_{k}\times\SS_{l}
\end{equation}
for~\(n\ge1\), where it is understood that \(k\),~\(l\ge0\).\footnote{Szczarba's parameter~\(r\) is \(k+1\) in our notation.}
For~\(n=1\) there is no choice as both domain and target are singletons.
For~\(n\ge2\) and~\(\ii=(i_{1},\dots,i_{n})\in\SS_{n}\) we define
\begin{equation}
  \Xi(\ii) = \bigl((\alpha,\beta),\jj,\kkk\bigr) \in \Shuff(k,l)\times\SS_{k}\times\SS_{l}
\end{equation}
recursively as follows:
For~\(\ii'=(i_{2},\dots,i_{n})\in\SS_{n-1}\) we already have
\begin{equation}
  \Xi(\ii')=\bigl((\alpha',\beta'),\jj',\kkk'\bigr) \in \Shuff(k',l')\times\SS_{k'}\times\SS_{l'}.
\end{equation}
If \(0\le i_{1}\le k'\), we set
\begin{gather}
  k=k'+1,
  \quad
  l = l',
  \quad
  \alpha=\{\,1,\alpha'_{1}+1,\dots,\alpha'_{k'}+1\,\},
  \quad
  \beta=\beta',
  \\* \notag
  \jj=(i_{1},\jj'),
  \quad
  \kkk=\kkk'.
\end{gather}
If \(k'<i_{1}\le n-1\), then we analogously define
\begin{gather}
  k=k',
  \quad
  l = l'+1,
  \quad
  \alpha=\alpha',
  \quad
  \beta=\{\,1,\beta'_{1}+1,\dots,\beta'_{l'}+1\,\},
  \\* \notag
  \jj=\jj',
  \quad
  \kkk=(i_{1}-k'-1,\kkk').
\end{gather}

\begin{proposition}[Szczarba, Hess--Tonks]
  \label{thm:szczarba-d}
  For any~\(n\ge0\),~\(\ii=(i_{1},\dots,i_{n})\in\SS_{n}\) and~\(x\in X_{n+1}\) one has
  \begin{align*}
    d_{0}\Sz_{\ii}x &=\Sz_{(i_{2},\dots,i_{n})}d_{i_{1}+1}x, \\
    d_{k}\Sz_{\ii}x &= d_{k}\Sz_{(i_{1},\dots,i_{k+1},i_{k}-1,\dots,i_{n})}x \qquad\text{for~\(0<k<n\) if \(i_{k}>i_{k+1}\),}\\
    d_{n}\Sz_{\ii}x &= s_{\beta-1}\Sz_{\jj}x(0\dots k+1)\cdot s_{\alpha-1}\Sz_{\kkk}x(k+1\dots n+1),
  \end{align*}
  where \(\Xi(\ii)=\bigl((\alpha,\beta),\jj,\kkk\bigr) \in \Shuff(k,l)\times\SS_{k}\times\SS_{l}\).
\end{proposition}

\begin{proof}
  These identities, implicit in the proof of~\cite[Thm.~2.1]{Szczarba:1961}, have been made
  explicit in~\cite[Lemma~6]{HessTonks:2006}. Note that in the latter reference
  the subscripts of the degeneracy operators~\(s_{\mu}\) and~\(s_{\nu}\)
  (which are \(s_{\alpha-1}\) and \(s_{\beta-1}\) in our notation) should be swapped.
\end{proof}

In order to state the second result we recall from~\cite[App.~B]{Franz:szczarba2} the map
\begin{equation}
  \label{eq:def-Phi}
  \Phi\colon \SS_{n} \times [n] \to \SS_{n-1} \times [n-1],
  \qquad
  (\ii,p) \mapsto (\jj,q)
\end{equation}
for~\(n\ge1\), which is recursively defined via
\begin{equation}
  \left\{\;
  \begin{alignedat}{4}
    \jj &= (i_{1}-1,\jj'), &\;\;\;  q &= q'+1 & \qquad\text{if\ \ } p &< i_{1}, &\quad\; (\jj',q') &\coloneqq \Phi(\ii',p), \\
    \jj &= \ii', &\;\;\;  q &= 0 & \qquad\text{if\ \ } p & = \rlap{\(i_{1}\)\ \ or\ \ \(i_{1}+1\),} \\
    \jj &= (i_{1},\jj'), &\;\;\;  q &= q'+1 & \qquad\text{if\ \ } p &> i_{1}+1, &\quad\; (\jj',q') &\coloneqq \Phi(\ii',p-1),
  \end{alignedat}
  \right.
\end{equation}
where \(\ii'=(i_{2},\dots,i_{n})\).

The following result will be as crucial for us as \Cref{thm:szczarba-d}.
Besides, it shows that Szczarba's twisting cochain~\(\tSz\)
is actually well-defined on normalized chains.

\begin{proposition}
  \label{thm:szczarba-s}
  Let \(X\) be a simplicial set, and let \(x\in X_{n}\) for some~\(n\ge1\).
  Pick \(\ii\in\SS_{n}\) and~\(0\le p\le n\), and set \((\jj,q)=\Phi(\ii,p)\). Then we have
  \begin{equation*}
    \Sz_{\ii}\,s_{p}\,x = s_{q}\,\Sz_{\jj}\,x.
  \end{equation*}
\end{proposition}

\begin{proof}
  See~\cite[Prop.~B.2\,(i)]{Franz:szczarba2}.
\end{proof}

\section{Relating the indices}

For any~\(n\ge0\) we set up a bijection
\begin{equation}
  \PSI\colon \SS_{n}\to S_{n},
  \qquad
  \ii\mapsto \PSI(\ii)=\pi
\end{equation}
\cf~the map~\(\beta\) defined in~\cite[Sec.~3.3]{MinichielloRiveraZeinalian:2023}.
The base case~\(n=0\) is trivial since both~\(\SS_{0}\) and~\(S_{0}\) are singletons.
For~\(n>0\) we define \(\PSI(\ii)=\pi\) recursively via
\begin{equation}
  \label{eq:def-psi-inv}
  \pi(k) = \begin{cases}
    i_{1}+1 & \text{if \(k=1\),} \\
    \pi'(k-1) & \text{if \(k\ge2\) and \(\pi'(k-1)\le i_{1}\),} \\
    \pi'(k-1)+1 & \text{if \(k\ge2\) and \(\pi'(k-1)>i_{1}\)} \\
  \end{cases}
\end{equation}
for~\(\ii=(i_{1},\dots,i_{n})\in\SS_{n}\), where \(\pi'=\PSI(i_{2},\dots,i_{n})\).
Observe that \(\pi'\) is obtained from~\(\pi\) by removing the assignment~\(1\mapsto\pi(1)\).
Expressed differently, formula~\eqref{eq:def-psi-inv} says that
\(\pi(k)\) equals \(i_{k}+1\) plus the number of indices~\(j<k\) such that \(\pi(j)<\pi(k)\).
The inverse function~\(\PSI^{-1}\colon \pi\mapsto\ii\) is therefore given by
\begin{equation}
  \label{eq:def-psi}
  i_{k} = \pi(k) - \# \{\,1\le j\le k\mid \pi(j)\le\pi(k)\,\}
\end{equation}
for~\(1\le k\le n\).

\begin{example}
  We have
  \begin{align}
    \PSI(0,\dots,0) &= \id, \\
    \PSI(n-1,n-2,\dots,0) &= \begin{pmatrix}
      1 & 2 & \cdots & n \\
      n & n-1 & \cdots & 1
    \end{pmatrix}, \\
    \PSI(4,2,0,1,0) &= \begin{pmatrix}
      1 & 2 & 3 & 4 & 5 \\
      5 & 3 & 1 & 4 & 2
    \end{pmatrix}.
  \end{align}
\end{example}

\begin{lemma}
  \label{thm:parity}
  Let \(n\ge0\) and \(\ii\in\SS_{n}\). Then for~\(\pi=\PSI(\ii)\in S_{n }\) we have
  \begin{equation*}
    (-1)^{\deg{\pi}} = (-1)^{\deg{\ii}}
  \end{equation*}
\end{lemma}

\begin{proof}
  This follows by induction from the trivial case~\(n=0\):
  In the notation of~\eqref{eq:def-psi-inv}, assume that the claim holds for~\(\ii'\) and~\(\pi'\).
  Then modulo~\(2\) we have
  \begin{equation}
    \deg{\pi} = (\pi(1)-1) + \deg{\pi'} \equiv i_{1} + \deg{\ii'} = \deg{\ii}.
    \qedhere
  \end{equation}
\end{proof}

\begin{lemma}
  \label{thm:prop-psi-1}
  Assume \(n\ge1\) and write~\(\ii'=(i_{2},\dots,i_{n})\in\SS_{n-1}\).
  Then the permutation~\(\pi'=\PSI(\ii')\in S_{n-1}\) is obtained from~\(\pi=\PSI(\ii)\in S_{n}\)
  by removing the assignment \(1\mapsto\pi(1)=i_{1}+1\).
\end{lemma}

\begin{proof}
  This follows directly from~\eqref{eq:def-psi-inv}.
\end{proof}

\begin{lemma}
  \label{thm:prop-psi-k}
  Let \(n\ge1\), and let \(\ii=(i_{1},\dots,i_{n})\in\SS_{n}\) be such that \(i_{k}>i_{k+1}\) for some \(1\le k<n\).
  Then for the adjacent transposition~\(\tau=(k,k+1)\) one has
  \begin{equation*}
    \PSI(\ii)\circ\tau=\PSI(i_{1},\dots,i_{k-1},i_{k+1},i_{k}-1,i_{k+2},\dots,i_{n}).
  \end{equation*}
\end{lemma}

\begin{proof}
  Writing \(\pi=\PSI(\ii)\),
  the assumption means \(\pi(k)>\pi(k+1)\), as can be seen from the recursive definition~\eqref{eq:def-psi-inv}.
  The claim now follows from the same formula or from~\eqref{eq:def-psi}.
\end{proof}

Recall from \Cref{thm:Psi-k-l} the bijection
\begin{equation}
  \Psi_{k,l}\colon \Shuff(k,l)\times S_{k}\times S_{l} \to S_{k+l}
\end{equation}
for~\(k\),~\(l\ge0\).

\begin{proposition}
  \label{thm:sz-bij-s}
  Rewriting the bijection~\(\Xi\) from~\eqref{eq:def-Xi}
  in terms of permutations via the bijection~\(\PSI\) gives the bijection
  \begin{equation*}
    S_{n} \to \bigsqcup_{k+l=n-1} \Shuff(k,l)\times S_{k}\times S_{l},
    \qquad
    \pi \mapsto \Psi_{k,l}^{-1}(\tilde\pi) = \bigl((\alpha,\beta),\sigma,\tau\bigr)
  \end{equation*}
  for~\(n\ge1\),
  where \(k=\pi(n)-1\) and~\(l=n-\pi(n)\), and
  \(\tilde\pi\) is obtained from~\(\pi\) by removing the assignment~\(n\mapsto\pi(n)\).
  Hence \(\alpha\) contains the indices mapped by~\(\pi\) to values less than~\(\pi(n)\),
  and \(\beta\) contains the indices mapped to values greater than~\(\pi(n)\).
\end{proposition}

\begin{proof}
  There is nothing to prove for~\(n=1\). Now assume the claim proven for~\(n-1\).
  Let \(\ii\in\SS_{n}\) and set \(\pi=\PSI(\ii)\), \(\pi'=\PSI(\ii')\).
  Set \(k=\pi(n)\), \(l=n-k-1\) and \(((\alpha,\beta),\sigma,\tau)=\Psi_{k,l}^{-1}(\pi)\),
  and analogously for~\(\pi'\). We have to check that these parameters are related
  as described in the inductive definition of~\(\Xi\), taking the definition of~\(\PSI\) into account.

  If \(\pi(1)=i_{1}+1\le k'+1=\pi'(n-1)\), then \(k=\pi(n)-1=\pi'(n-1)=k'+1\).
  Hence \(1\in\alpha\) so that \(\alpha=\{1,\alpha'_{1}+1,\dots,\alpha'_{k'}+1\}\).
  Moreover, \(\sigma\) with the assignment~\(1\mapsto\sigma(1)\) removed equals \(\sigma'\),
  as required by the definition~\(\jj=(i_{1},\jj')\) of~\(\Xi\) together with \Cref{thm:prop-psi-1}.
  Also, \(\beta'=\beta\) and \(\tau=\tau'\).

  If \(\pi(1)=i_{1}+1>k'+1=\pi'(n-1)\), then \(k=\pi(n)-1=\pi'(n-1)-1=k'\).
  Hence \(1\in\beta\) so that \(\beta=\{1,\beta'_{1}+1,\dots,\beta'_{l'}+1\}\).
  Moreover, \(\tau\) with the assignment~\(1\mapsto\tau(1)\) removed equals \(\tau'\).
  Also, \(\alpha=\alpha'\) and \(\sigma=\sigma'\).
\end{proof}

\begin{example}
  If
  \begin{equation}
    \pi = \begin{pmatrix}
      1 & 2 & 3 & 4 & 5 \\
      5 & 3 & 1 & 4 & 2
    \end{pmatrix},
    \quad\text{then}\quad
    k=1,
    \quad
    l=3,
    \quad
    \tilde\pi = \begin{pmatrix}
      1 & 2 & 3 & 4 \\
      4 & 2 & 1 & 3
    \end{pmatrix},
  \end{equation}
  and
  \begin{equation}
    \alpha=\{3\},
    \quad
    \beta=\{1,2,4\},
    \quad
    \sigma=\begin{pmatrix}
      1 \\ 1
    \end{pmatrix},
    \quad
    \tau=\begin{pmatrix}
      1 & 2 & 3 \\
      3 & 1 & 2
    \end{pmatrix}.
  \end{equation}
\end{example}

We also need to translate the map~\(\Phi\) defined in~\eqref{eq:def-Phi} to the language of permutations.

\begin{proposition}
  \label{thm:sz-bij-Phi}
  Let \(\ii\in\SS_{n}\) and \(0\le p\le n\) where \(n\ge1\). Set \((\jj,q)=\Phi(\ii,p)\);
  also write \(\pi=\PSI(\ii)\in S_{n}\) and~\(\tilde\pi=\PSI(\jj)\in S_{n-1}\).
  \begin{enumroman}
  \item \label{p1} If \(p=0\), then \(q+1=\pi^{-1}(1)\), and \(\tilde\pi\) is obtained from~\(\pi\) by removing the assignment~\(q+1\mapsto1\).
  \item \label{p2} If \(0<p<n\), then \(q+1=\min\bigl(\pi^{-1}(p),\pi^{-1}(p+1)\bigr)\), and
    \(\tilde\pi\) is obtained from~\(\pi\) by removing the assignment~\(q+1\mapsto\pi(q+1)\in\{p,p+1\}\).
  \item \label{p3} If \(p=n\), then \(q+1=\pi^{-1}(n)\), and \(\tilde\pi\) is obtained from~\(\pi\) by removing the assignment~\(q+1\mapsto n\).
  \end{enumroman}
\end{proposition}

\begin{proof}
  We prove the claims by induction on~\(n\), the case~\(n=1\) being trivial.

  \ref{p1}: If \(i_{1}=0\), then \(\pi(1)=1\), \(q=0\) and \(\jj=\ii'\). Hence
  \(\tilde\pi=\PSI(\ii')\) is obtained from~\(\pi\) by removing the assignment~\(q+1=1\mapsto\pi(1)=1\).

  Now assume \(i_{1}>0\). Then \(\jj=(i_{1}-1,\jj')\) and~\(q=q'+1\) where \((\jj',q')=\Phi(\ii',1)\).
  The permutation~\(\pi'\coloneqq\PSI(\ii')\) is obtained from~\(\pi\) by removing the assignment~\(1\mapsto\pi(1)=i_{1}+1>1\).
  By induction, we have \(q'+1=(\pi')^{-1}(1)\), and also that
  \(\tilde\pi'\coloneqq\PSI(\jj')\) is obtained from~\(\pi'\) by removing the assignment~\(q'+1\mapsto1\).
  Hence \(\tilde\pi\) is obtained from~\(\tilde\pi'\) by adding the assignment~\(1\mapsto i_{1}=\pi(1)-1\).
  Equivalently, it is obtained from~\(\pi\) by removing the assignment~\(q+1\mapsto1\).

  \ref{p3}: This is analogous to the previous case. If \(i_{1}=n-1\), then \(\pi(1)=n\), \(q=0\) and \(\jj=\ii'\).
  This means that \(\tilde\pi\) is obtained from~\(\pi\) by removing the assignment~\(q+1=1\mapsto\pi(n)=1\).

  If \(i_{1}<n-1\), then \(\jj=(i_{1},\jj')\) and~\(q=q'+1\) where \((\jj',q')=\Phi(\ii',n-1)\).
  The permutation~\(\pi'\coloneqq\PSI(\ii')\in S_{n-1}\) is obtained from~\(\pi\)
  by removing the assignment \(1\mapsto\pi(1)=i_{1}+1<n\).
  By induction, we have \(q'+1=(\pi')^{-1}(n-1)\), and also that
  \(\tilde\pi'\coloneqq\PSI(\jj')\) is obtained from~\(\pi'\) by removing the assignment~\(q'+1\mapsto n-1\).
  Hence \(\tilde\pi\) is obtained from~\(\tilde\pi'\) by adding the assignment~\(1\mapsto i_{1}+1=\pi(1)\).
  Equivalently, it is obtained from~\(\pi\) by removing the assignment~\(q+1\mapsto n\).

  \ref{p2}: If \(p\in\{i_{1},i_{1}+1\}=\{\pi(1)-1,\pi(1)\}\), then \(q=0\) and~\(\jj=\ii'\).
  Hence \(\tilde\pi\) is obtained from~\(\pi\) by removing the assignment~\(1\mapsto\pi(1)\).
  Moreover, \(q+1=1=\min(\pi^{-1}(p),\pi^{-1}(p+1))\) as claimed.

  If \(p<i_{1}\), then \(\jj=(i_{1}-1,\jj')\) and~\(q=q'+1\) where \((\jj',q')=\Phi(\ii',p)\).
  The permutation~\(\pi'=\PSI(\ii')\) is obtained from~\(\pi\) by removing the assignment~\(1\mapsto\pi(1)=i_{1}+1>p+1\).
  By induction, we have
  \begin{equation}
    q'+1 = \min\bigl((\pi')^{-1}(p),(\pi')^{-1}(p+1)\bigr),
  \end{equation}
  and also that \(\tilde\pi'=\PSI(\jj')\) is obtained from~\(\pi'\) by removing
  the assignment~\(q'+1\mapsto\pi'(q'+1)\le p+1<\pi(1)\). Hence
  \begin{equation}
    q+1 = \min\bigl(\pi^{-1}(p),\pi^{-1}(p+1)\bigr),
  \end{equation}
  and \(\tilde\pi\) is obtained from~\(\tilde\pi'\) by adding the assignment~\(1\mapsto\pi(1)-1\).
  Equivalently, it is obtained from~\(\pi\) by removing the assignment~\(q+1\mapsto\pi'(q'+1)=\pi(q+1)\).

  If \(p>i_{1}+1\), then \(\jj=(i_{1},\jj')\) and~\(q=q'+1\) where \((\jj',q')=\Phi(\ii',p-1)\).
  The permutation~\(\pi'=\PSI(\ii')\) is obtained from~\(\pi\) by removing the assignment~\(1\mapsto\pi(1)=i_{1}+1<p\).
  By induction, we have
  \begin{equation}
    q'+1 = \min\bigl((\pi')^{-1}(p-1),(\pi')^{-1}(p)\bigr),
  \end{equation}
  and also that \(\tilde\pi'=\PSI(\jj')\) is obtained from~\(\pi'\) by removing
  the assignment~\(q'+1\mapsto\pi'(q'+1)\ge p-1\ge\pi(1)\). Hence
  \begin{align}
    q+1 &= q'+2 = \min\bigl((\pi')^{-1}(p-1)+1,(\pi')^{-1}(p)+1\bigr) \\
    \notag &= \min\bigl(\pi^{-1}(p),\pi^{-1}(p+1)\bigr),
  \end{align}
  and \(\tilde\pi\) is obtained from~\(\tilde\pi'\) by adding the assignment~\(1\mapsto\pi(1)\).
  Equivalently, it is obtained from~\(\pi\) by removing the assignment~\(q+1\mapsto\pi'(q'+1)=\pi(q+1)\).
\end{proof}

It will be convenient to change the indexing of the Szczarba operators
from sequences~\(\ii\in\SS_{n}\) to permutations~\(\pi\in S_{n}\), where \(p\colon\SS_{n}\to S_{n}\)
in the bijection introduced in~\eqref{eq:def-psi-inv}. We therefore write
\begin{equation}
  \Sz_{\pi}x = \Sz_{\ii}x
\end{equation}
for~\(\pi=p(\ii)\) and~\(x\in X_{n+1}\).
We state the properties of the Szczarba operators discussed so far in this new notation.

\begin{proposition}
  \label{thm:properties-sz-new-d}
  For~\(x\in X_{n+1}\), \(\tilde\pi\in S_{n}\) and \(1\le i\le n\) we have the following:
  \begin{enumroman}
  \item \label{q1}
    If \(\tilde\pi\) is obtained from~\(\pi\in S_{n-1}\) by adding the assignment~\(1\mapsto i\), then
    \begin{equation*}
      d_{0}\,\Sz_{\tilde\pi}x = \Sz_{\pi} d_{i}\,x.
    \end{equation*}
  \item \label{q2}
    If \(\tilde\pi=\pi\circ(j,j+1)\) for some~\(\pi\in S_{n}\) and~\(0<j<n\), then
    \begin{equation*}
      d_{j}\,\Sz_{\tilde\pi}x = d_{j}\,\Sz_{\pi}x.
    \end{equation*}
  \item \label{q3}
    If \(\tilde\pi\) is obtained from~\(\pi\in S_{n-1}\) by adding the assignment~\(n\mapsto i\), then
    \begin{equation*}
      d_{n}\,\Sz_{\tilde\pi}x = s_{\beta-1}\Sz_{\sigma} x(0\dots i)\cdot s_{\alpha-1}\Sz_{\tau} x(i\dots n+1),
    \end{equation*}
    where \(((\alpha,\beta),\sigma,\tau)=\Psi_{i-1,n-i}^{-1}(\pi)\).
  \end{enumroman}
\end{proposition}

Here \((j,j+1)\in S_{n}\) is a transposition, and \(\Psi_{i-1,n-i}\) refers to the bijection from \Cref{thm:Psi-k-l}.
The product in the last displayed formula above is taken in the group~\(G_{n-1}\).

\begin{proof}
  This follows by combining the identities in \Cref{thm:szczarba-d} with \Cref{thm:prop-psi-1},
  \Cref{thm:prop-psi-k} and \Cref{thm:sz-bij-s}, respectively.
\end{proof}

\begin{proposition}
  \label{thm:properties-sz-new-s}
  For~\(x\in X_{n+1}\),~\(\pi\in S_{n+1}\) and~\(\tilde\pi\in S_{n}\) we have the following:
  \begin{enumroman}
  \item \label{qq1}
    If \(\tilde\pi\) is obtained from~\(\pi\) by removing the assignment~\(j\mapsto1\), then
    \begin{equation*}
      s_{j-1}\Sz_{\tilde\pi} x = \Sz_{\pi}s_{0}\,x.
    \end{equation*}
  \item \label{qq2}
    Given \(1\le i\le n\), set \(j=\min\bigl(\pi^{-1}(i),\pi^{-1}(i+1)\bigr)\).
    If \(\tilde\pi\) is obtained from~\(\pi\) by removing the assignment~\(j\mapsto\pi(j)\), then
    \begin{equation*}
      s_{j-1}\Sz_{\tilde\pi} x = \Sz_{\pi}s_{i}\,x.
    \end{equation*}
  \item \label{qq3}
    If \(\tilde\pi\) is obtained from~\(\pi\) by removing the assignment~\(j\mapsto n+1\), then
    \begin{equation*}
      s_{j-1}\Sz_{\tilde\pi} x = \Sz_{\pi}s_{n+1}\,x.
    \end{equation*}
  \end{enumroman}
\end{proposition}

\begin{proof}
  All claims follow by combining \Cref{thm:szczarba-s} with \Cref{thm:sz-bij-Phi}.
\end{proof}

\section{The cubical cobar construction}

Let \(X\) be a \(1\)-reduced simplicial set.
We want to define a certain cubical set associated to~\(X\), called the cubical cobar construction~\(\bbOM X\).
We refer to~\cite[Prop.~4.2]{RiveraZeinalian:2018} or~\cite[Sec.~3.6]{MMRivera:2024} for a less pedestrian way.

We start by defining the following graded set~\(Q(X)\). An element of~\(Q(X)\) of dimension~\(n\) is a pair consisting
of a cube~\(x\in X_{m}\), \(m\ge1\), and a subset~\(I\subset\{2,...,n\}\) such that \(m+\deg{I}=n+1\).
We call such a subset of~\(\{1,\dots,n+1\}\) \newterm{inner}.
It keeps track of formal degeneracy operators, see \eqref{eq:cobarel-x-I-s-seq} below.
We also write the pair~\((x,I)\) as~\(\CobarEl{x}_{I}\) and abbreviate \(\CobarEl{x}_{\emptyset}\) to~\(\CobarEl{x}\).

On~\(Q(X)\) we define degeneracy operators as follows: For~\(\CobarEl{x}\) as before we set
\begin{equation}
  s_{1}\CobarEl{x} = \CobarEl{s_{0}\,x},
  \quad
  s_{i}\CobarEl{x} = \CobarEl{x}_{\{i\}} \quad \text{for~\(2\le i\le m\),}
  \quad
  s_{m+1}\CobarEl{x} = \CobarEl{s_{m+1}\,x}.
\end{equation}
This is extended to all~\(\CobarEl{x}_{I}\) in a way that leads to the cubical relation~\eqref{eq:cube-ss}. Explicitly,
\begin{equation}
  s_{1}\CobarEl{x}_{I} = \CobarEl{s_{0}\,x}_{I+1},
  \quad
  s_{i}\CobarEl{x}_{I} = \CobarEl{x}_{I'} \;\;\text{for~\(2\le i\le n\),}
  \quad
  s_{n+1}\CobarEl{x}_{I} = \CobarEl{s_{m+1}\,x}_{I},
\end{equation}
where \(I+1\) denotes the set~\(\{\,j+1\mid j\in I\,\}\) and
\begin{equation}
  I' = \{\,j\in I\mid j<i\,\} \cup \{i\} \cup \{\,j+1\in I\mid j>i\,\}.
\end{equation}
Writing the elements of~\(I\) as~\(2\le i_{1}<\dots<i_{k}\le n\), we then have
\begin{equation}
  \label{eq:cobarel-x-I-s-seq}
  \CobarEl{x}_{I} = s_{i_{k}}\cdots s_{i_{1}}\CobarEl{x}.
\end{equation}

Next we define the graded set
\begin{equation}
  P(X) = \bigsqcup_{k\ge0} Q(X)^{k}\!\bigm/\sim
\end{equation}
where the identification is done as in the definition~\eqref{eq:def-product-cubical} of the product of cubical sets
(for all adjacent factors simultaneously). We write elements of~\(P(X)\) in the form~\([y_{1},\dots,y_{k}]\)
with~\(y_{1}\),~\dots,~\(y_{k}\in Q(X)\).
Concatenating them gives an associative product with the unique element of length~\(k=0\) as identity element.

We finally define the graded set
\begin{equation}
  \bbOM X = P(X)/{\sim}
\end{equation}
where this time the equivalence relation is generated by
\begin{equation}
  \bigl[y_{1},\dots,y_{k}] \sim \bigl[y_{1},\dots,y_{j-1},y_{j+1},\dots,y_{k}\bigr]
\end{equation}
whenever a component~\(y_{j}\) has dimension~\(0\)
(so that \(y_{j}=\CobarEl{*_{1}}\) for the unique \(1\)-simplex~\(*_{1}\) in the \(1\)-reduced simplicial set~\(X\)).
We write the element of~\(\bbOM X\) represented by~\([\CobarEl{x_{1}},\dots,\CobarEl{x_{k}}]\in P(X)\) as
\begin{equation}
  \CobarEl{x_{1},\dots,x_{k}} \in \bbOM X.
\end{equation}
The multiplication in~\(P(X)\) descends to~\(\bbOM X\).

We now turn \(\bbOM X\) into a cubical monoid by adding the cubical structure maps.
For~\(\CobarEl{x}\in Q(X)\) as before and~\(1\le i\le m\) we define
\begin{align}
  d^{1}_{i} \CobarEl{x} &= \CobarEl{d_{i}\,x}, \\
  \gamma_{i} \CobarEl{x} &= \CobarEl{s_{i}\,x}.
\end{align}
We extend this to all of~\(Q(X)\) using the cubical relations listed in \Cref{sec:cubical}. We then extend these operators as well as
the degeneracy operators~\(s_{i}\) to~\(P(X)\) as done for a product of cubical sets in~\eqref{eq:def-product-cubical}.
From there they descend to~\(\bbOM X\).

We add the remaining face operators~\(d^{0}_{i}\) on~\(\bbOM X\)
via maps defined on~\(Q(X)\). For~\(\CobarEl{x}\in Q(X)\) as before and \(1\le i\le m\) we set
\begin{equation}
  d^{0}_{i}\CobarEl{x} = \bigCobarEl{x(0\dots i), x(i\dots m+1)} \in \bbOM X.
\end{equation}
We again extend this to all of~\(Q(X)\) using the cubical relations. These extensions induce maps~\(P(X)\to\bbOM X\),
which descend to the face operators~\(d^{0}_{i}\colon\bbOM X\to\bbOM X\).
This turns \(\bbOM X\) into a cubical set, called the \newterm{cubical cobar construction} of~\(X\).
The multiplication is compatible with this structure, so that we in fact obtain a cubical monoid.

\medbreak

Let \(\OM\,C(X)\) be the (reduced) cobar construction of the dgc~\(C(X)\).
The result we recall below will justify that
we write elements of~\(\OM\,C(X)\) in the form~\(\CobarEl{x_{1},\dots,x_{k}}\),
where \(x_{1}\),~\dots,~\(x_{k}\) are positive-dimensional simplices in~\(X\).
The cobar construction is a differential graded algebra (dga) with concatenation of elements as product.
Using the homotopy Gerstenhaber structure of~\(C(X)\), one can define a diagonal on~\(\OM\,C(X)\)
that is compatible with the product, so that one actually obtains a differential graded (dg) bialgebra.
(See \cite[\S 3]{Franz:szczarba2} for details.) For our purposes, everything we need to know
about~\(\OM\,C(X)\) is contained in the following statement, compare~\cite[Thm.~3.3]{MMRivera:2024}.\footnote{%
  Contrary to what is stated in~\cite[Sec.~2.4]{MMRivera:2024}, one needs to divide
  out both degeneracies and connections in the definition of~\(C(\bbOM X)\) for~\cite[Thm.~3.3]{MMRivera:2024} to hold.}

\begin{proposition}
  \label{thm:iso-cobars}
  Let \(X\) be a \(1\)-reduced simplicial set.
  The ``identity map'' sending \(\CobarEl{x_{1},\dots,x_{k}}\in\OM\,C(X)\) to~\(\CobarEl{x_{1},\dots,x_{k}}\in C(\bbOM X)\)
  is an isomorphism of dg~bialgebras.
\end{proposition}

\begin{proof}
  The definition of the structure maps for~\(\bbOM X\) entails
  that a cube~\(\CobarEl{x_{1},\dots,x_{k}}\in\bbOM X\) is neither degenerate nor folded
  if and only if all~\(x_{i}\) are non-degenerate simplices, that is, if and only if
  \(\CobarEl{x_{1},\dots,x_{k}}\in\OM\,C(X)\) is non-zero. Hence we get an isomorphism
  of graded \(\kk\)-modules and in fact of graded \(\kk\)-algebras because both products
  are defined via concatenation.

  For an \(n\)-dimensional cube~\(\CobarEl{x}\) we have
  \begin{align}
    d^{0}_{1}\CobarEl{x} &= \bigCobarEl{x(0\dots1),x(1\dots n+1)} = \CobarEl{*_{1},d_{0}\,x} = \CobarEl{d_{0}\,x}, \\
    d^{0}_{n}\CobarEl{x} &= \bigCobarEl{x(0\dots n),x(n\dots n+1)} = \CobarEl{d_{n+1}\,x,*_{1}} = \CobarEl{d_{n+1}\,x}.
  \end{align}
  On the other hand, given that \(X\) is \(1\)-reduced, the reduced diagonal
  \begin{equation}
    \Deltabar\colon \bar C(X) \hookrightarrow C(X) \xrightarrow{\Delta} C(X)\otimes C(X) \twoheadrightarrow \bar C(X) \otimes \bar C(X)
  \end{equation}
  in the reduced chain complex~\(\bar C(X)\) takes the form
  \begin{equation}
    \Deltabar\,x = \sum_{i=2}^{n-1} x(0\dots i) \otimes x(i\dots n+1).
  \end{equation}
  Taken together, these two facts imply that the differential
  of~\(\CobarEl{x}\) in the cubical chain complex~\(C(\bbOM X)\)
  agrees with the differential of~\(\CobarEl{x}=\desusp x\in\OM\,C(X)\),
  \begin{align}
    d_{C(\bbOM X)}\CobarEl{x} &= \sum_{i=1}^{n}(-1)^{i}\,d^{0}_{i}\CobarEl{x} - \sum_{i=1}^{n}(-1)^{i}\,d^{1}_{i}\CobarEl{x} \\
    \notag &= - \CobarEl{d_{0}\,x} + \sum_{i=2}^{n-1}(-1)^{i}\,\bigCobarEl{x(0,\dots,i),x(i,\dots,n+1)} \\
    \notag &\qquad + (-1)^{n}\,\CobarEl{d_{n+1}\,x} - \sum_{i=1}^{n}(-1)^{i}\,\CobarEl{d_{i}\,x} \\
    \notag &= - \CobarEl{d\,x} + (\desusp\otimes\desusp)\,\Deltabar\,x = d_{\OM\,C(X)}\CobarEl{x}.
  \end{align}
  (Here \(\desusp\) denotes the desuspension operator.)

  It follows that the differentials agree on all elements because they are compatible with the products.
  The diagonal on~\(C(\bbOM X)\) is the ``geometric diagonal'' considered by Baues~\cite[Sec.~IV.2]{Baues:1980}.
  As explained in~\cite[App.~A]{Franz:szczarba2}, under the ``identity'' isomorphism
  it agrees with the one on~\(\OM\,C(X)\) induced by the homotopy Gerstenhaber structure on~\(C(X)\).
\end{proof}

\section{The main result}

The goal of this section is to prove the following more precise version of \Cref{thm:intro:main}.

\begin{theorem}
  \label{thm:main}
  Let \(X\) be a \(1\)-reduced simplicial set and \(G\) a simplicial group,
  and let \(\tau\colon X_{>0}\to G\) be a twisting function.
  \begin{enumroman}
  \item \label{thm:main-1}
    The assignment
    \begin{equation*}
      \II_{\CobarEl{x}} \ni u_{\pi} \mapsto \Sz_{\pi}x \in G
    \end{equation*}
    for~\(x\in X_{n+1}\) and~\(\pi\in S_{n}\), \(n\ge0\),
    induces a morphism of simplicial monoids
    \begin{equation*}
      \FFF\colon\TT\bbOM X\to G.
    \end{equation*}
  \item \label{thm:main-2}
    The composition
    \begin{equation*}
      \OM\,C(X) \cong C(\bbOM X) \stackrel{\tt}{\longrightarrow} C(\TT \bbOM X) \stackrel{C(\FFF)}{\longrightarrow} C(G),
    \end{equation*}
    is the Szczarba map~\(\fSz\) induced by Szczarba's twisting cochain~\(\tSz\).
  \end{enumroman}
\end{theorem}

Our strategy for proving \Cref{thm:main} is to define a simplicial map
\begin{equation}
  \label{eq:f-z}
  \FF_{z}\colon \II_{z} \to G,
  \qquad
  u \mapsto \FF(z;u)
\end{equation}
for each cube~\(z\in\bbOM X\) and then to verify that these maps induce the desired map~\(\FFF\),
\cf~\Cref{thm:triangulation-map}. Multiplicativity will be clear from the construction.
We construct the map~\eqref{eq:f-z} in several steps.

For an \((n+1)\)-simplex~\(x\in X\) and an \(n\)-simplex~\(u_{\pi}\) corresponding to a permutation~\(\pi\in S_{n}\) we of course set
\begin{equation}
  \FF(\CobarEl{x};u_{\pi}) = \Sz_{\pi}x.
\end{equation}
We then have
\begin{equation}
  d_{j}\,\FF\bigl(\CobarEl{x};u_{\pi}\bigr) = d_{j}\,\FF\bigl(\CobarEl{x};u_{\pi\circ(j,j+1)}\bigr)
\end{equation}
for all~\(0<j<n\) by \Cref{thm:properties-sz-new-d}\,\ref{q2}.
It therefore follows from \Cref{thm:triang-map} that these assignments induce a simplicial map
\begin{equation}
  \FF_{\CobarEl{x}}\colon \II_{\CobarEl{x}} \to G
  \qquad
  u \mapsto \FF(\CobarEl{x};u)
\end{equation}
for each~\(\CobarEl{x}\in Q(X)\).
We extend \(\FF\) to a map defined on triangulations of all cubes~\(\CobarEl{x}_{I}\in Q(X)\) by setting
\begin{equation}
  \label{eq:def-F-x-I}
  \FF(\CobarEl{x}_{I};u) = \FF\bigl(\CobarEl{x};(\sigma_{I})_{*}u\bigr)
\end{equation}
for any~\(x\) as before and
any inner interval~\(I\subset\{1,\dots,n\}\).
In total we get a map defined on the disjoint union of these simplicial cubes,
\begin{equation}
  \FF \colon \bigsqcup_{\CobarEl{x}_{I}\in Q(X)} \II_{\CobarEl{x}_{I}} \to G,
  \qquad
  \II_{\CobarEl{x}_{I}} \ni u \mapsto \FF\bigl(\CobarEl{x}_{I};u\bigr). 
\end{equation}

\begin{lemma}
  \label{thm:F-s-CobarEl-x-u}
  For any~\(\CobarEl{x}_{I}\in Q(X)\) and any simplex~\(u\in\II_{\CobarEl{x}_{I}}\) we have
  \begin{equation*}
    \FF\bigl(s_{i}\CobarEl{x}_{I};u\bigr)=\FF\bigl(\CobarEl{x}_{I};(\sigma_{i})_{*}\,u\bigr).
  \end{equation*}
\end{lemma}

\begin{proof}
  For later reference we prefer to write \(\sigma_{i}=\lambda\) in part of the proof.

  We first assume \(I=\emptyset\) and \(u=u_{\pi}\) for some~\(\pi\in S_{n}\) where \(n=\deg{x}-1\).
  For~\(1<i<n+1\) the claimed identity is a special case of the definition~\eqref{eq:def-F-x-I}
  since \(s_{i}\CobarEl{x}=\CobarEl{x}_{\{i\}}\).

  In the case~\(i=1\) let \(\tilde\pi\) be obtained from~\(\pi\) by removing the assignment~\(j\coloneqq\pi^{-1}(1)\mapsto 1\).
  Using \Cref{thm:properties-sz-new-s} and \Cref{thm:s-u-tilde-pi}, we have
  \begin{align}
    \FF(s_{1}\CobarEl{x};u_{\pi}) &= \FF(\CobarEl{s_{0}\,x};u_{\pi}) = \Sz_{\pi} s_{0}\,x = s_{j-1}\,\Sz_{\tilde\pi}\,x \\
    \notag &= s_{j-1}\,\FF(\CobarEl{x};u_{\tilde\pi}) = \FF\bigl(\CobarEl{x};s_{j-1}\,u_{\tilde\pi}\bigr) = \FF\bigl(\CobarEl{x};(\sigma_{1})_{*}u_{\pi}\bigr).
  \end{align}

  For~\(i=n+1\) let \(\tilde\pi\) be obtained from~\(\pi\) by removing the assignment~\(j\coloneqq\pi^{-1}(n+1)\mapsto n+1\).
  Analogously to the previous case, we have
  \begin{align}
    \FF(s_{n+1}\CobarEl{x};u_{\pi}) &= \FF(\CobarEl{s_{n+1}\,x};u_{\pi}) = \Sz_{\pi} s_{n+1}\,x = s_{j-1}\,\Sz_{\tilde\pi}\,x \\
    \notag &= s_{j-1}\,\FF(\CobarEl{x};u_{\tilde\pi}) = \FF\bigl(\CobarEl{x};s_{j-1}\,u_{\tilde\pi}\bigr) = \FF\bigl(\CobarEl{x};(\sigma_{n+1})_{*}u_{\pi}\bigr).
  \end{align}

  We now consider the case where \(I=\emptyset\) is as before, but \(u\in\II_{\lambda^{*}\CobarEl{x}}\) is arbitrary.
  Since simplices of the form~\(u_{\pi}\) generate the triangulation of a cube
  (see \Cref{thm:cube-generated-embedding}\,\ref{thm:cube-generated-embedding-1})
  we have \(u=\mu^{*}u_{\pi}\) for some permutation~\(\pi\) and some simplicial operator~\(\mu^{*}\).
  Using the fact that both~\(\lambda_{*}\) and~\(\FF\) are maps of simplicial sets, we get
  \begin{align}
    \FF(\lambda^{*}\CobarEl{x};u) &= \FF(\lambda^{*}\CobarEl{x};\mu^{*}u_{\pi})
    = \mu^{*} \FF(\lambda^{*}\CobarEl{x};u_{\pi})
    = \mu^{*} \FF(\CobarEl{x};\lambda_{*} u_{\pi}) \\
    \notag &= \FF(\CobarEl{x};\mu^{*} \lambda_{*} u_{\pi})
    = \FF(\CobarEl{x};\lambda_{*}\,\mu^{*} u_{\pi})
    = \FF(\CobarEl{x};\lambda_{*} u).
  \end{align}

  We finally look at the case of an arbitrary inner interval~\(I\) and an arbitrary simplex~\(u\in\II_{\lambda^{*}\CobarEl{x}}\).
  By the transpose of the cubical relation~\eqref{eq:cube-ss} (or directly the definition~\eqref{eq:cubecat:s} of~\(\sigma_{i}\)),
  we have \(\sigma_{I}\lambda=\lambda'\sigma_{I'}\) for some inner subset~\(I'\),
  where \(\lambda'\) is either some~\(\sigma_{i'}\) or the identity map.
  Hence
  \begin{align}
    \FF(\lambda^{*}\CobarEl{x}_{I};u) &= \FF(\lambda^{*}\sigma_{I}^{*}\CobarEl{x};u) = \FF(\sigma_{I'}^{*}(\lambda')^{*}\CobarEl{x};u) \\
    \notag &= \FF((\lambda')^{*}\CobarEl{x};(\sigma_{I'})_{*}u) = \FF(\CobarEl{x};(\lambda')_{*}(\sigma_{I'})_{*}u) \\
    \notag &= \FF(\CobarEl{x};(\sigma_{I})_{*}\lambda_{*}u) = \FF(\sigma_{I}^{*}\CobarEl{x};\lambda_{*}u) = \FF(\CobarEl{x}_{I};\lambda_{*}u),
  \end{align}
  which completes the proof.
\end{proof}

We now extend \(\FF\) to triangulations of cubes in the \(k\)-fold Cartesian set product of~\(Q(X)\) via
\begin{equation}
  \label{thm:def-F-product}
  \FF\bigl(y_{1},\dots,y_{k};(u_{1},\dots,u_{k})\bigr) = \FF(y_{1};u_{1})\cdots \FF(y_{k};u_{k}),
\end{equation}
where \(u=(u_{1},\dots,u_{k})\) is, say, an \(m\)-simplex
in~\(\II_{y_{1}}\times\dots\times\II_{y_{k}}\).
The product on the right-hand side is taken in the group~\(G_{m}\).
It is empty for~\(k=0\), so that \(u\in\II^{0}\) is sent to~\(1\in G_{m}\) in this case.

\begin{lemma}
  The assignment~\eqref{thm:def-F-product} descends to a map
  \begin{align*}
    \FF\colon \bigsqcup_{[y_{1},\dots,y_{k}]\in P(X)} \II_{[y_{1},\dots,y_{k}]} &\to G, \\
    \bigl([y_{1},\dots,y_{k}],(u_{1},\dots,u_{k})\bigr) &\mapsto \FF(y_{1};u_{1})\cdots \FF(y_{k};u_{k}).
  \end{align*}
\end{lemma}

\begin{proof}
  Consider \(y\),~\(y'\in Q(X)\), say with \(\deg{y}=k\) and~\(\deg{y'}=l\).
  By \Cref{thm:F-s-CobarEl-x-u} together with the definition of~\((\sigma_{i})_{*}\) as the projection map dropping the \(i\)-th factor,
  we have for any simplex~\(u=(u_{1},\dots,u_{k+l+1})\in\II^{k+l+1}\) the identity
  \begin{align}
    \FF\bigl((s_{k+1}y, y');u\bigr) &= \FF\bigl(s_{k+1}y;(u_{1},\dots,u_{k+1})\bigr) \cdot \FF\bigl(y';(u_{k+2},\dots,u_{k+l+1})\bigr) \\
    \notag &= \FF\bigl(y;(\sigma_{k+1})_{*}(u_{1},\dots,u_{k+1})\bigr) \cdot \FF\bigl(y';(u_{k+2},\dots,u_{k+l+1})\bigr) \\
    \notag &= \FF\bigl(y;(u_{1},\dots,u_{k})\bigr) \cdot \FF\bigl(y';(u_{k+2},\dots,u_{k+l+1})\bigr) \\
    \notag &= \FF\bigl(y;(u_{1},\dots,u_{k})\bigr) \cdot \FF\bigl(y';(\sigma_{1})_{*}(u_{k+1},\dots,u_{k+l+1})\bigr) \\
    \notag &= \FF\bigl(y;(u_{1},\dots,u_{k})\bigr) \cdot \FF\bigl(s_{1}y';(u_{k+1},\dots,u_{k+l+1})\bigr) \\
    \notag &= \FF\bigl((y, s_{1}y');u\bigr).
  \end{align}
  The case of more than two factors is analogous.
\end{proof}

The map~\(\FF\) passes from triangulations of cubes in~\(P(X)\) to those of cubes in~\(\bbOM X\)
because we have \(\FF(y;u)=1\in G\) whenever \(y\in Q(X)\) is a \(0\)-cube.
As a result, we now have a map
\begin{equation}
  \FF\colon \bigsqcup_{z\in\bbOM X} \II_{z} \to G.
\end{equation}
To obtain the map~\(f\colon\TT\bbOM X\to G\),
we must show that \(F\) is compatible with the identifications
made in the definition of the triangulation~\(\TT\bbOM X\).
So we need
\begin{equation}
  \label{eq:F-lambda-z-u}
  \FF(\lambda^{*}z;u)=\FF(z;\lambda_{*}u)
\end{equation}
to hold for any \(n\)-cube~\(z\in\bbOM X\), any~\(\lambda\colon\TWO_{\lambda^{*}z}\to\TWO_{z}\) and any simplex~\(u\in\II_{\lambda^{*}z}\).
In the special case~\(z=\CobarEl{x}_{I}\) and~\(\lambda=\sigma_{i}\), this has already been achieved in \Cref{thm:F-s-CobarEl-x-u}.

\begin{lemma}
  \label{thm:F-d-gamma-CobarEl-x-u-pi}
  Let \(\CobarEl{x}_{I}\in Q(X)\), and let \(\lambda=\delta^{\epsilon}_{i}\) or~\(\lambda=\gamma_{i}\) for some~\(i\) and~\(\epsilon\).
  Then for any simplex~\(u\in\II_{\lambda^{*}\CobarEl{x}_{I}}\) we have
  \begin{equation*}
    \FF\bigl(\lambda^{*}\CobarEl{x}_{I};u\bigr)=\FF\bigl(\CobarEl{x}_{I};\lambda_{*}u\bigr).
  \end{equation*}
\end{lemma}

\begin{proof}
  As in the proof of \Cref{thm:F-s-CobarEl-x-u}, we start by assuming \(I=\emptyset\)
  and \(u=u_{\pi}\) for some~\(\pi\in S_{n}\), where \(n=\deg{x}-1\).

  Under these assumptions, we first consider the case~\(\lambda=\delta^{1}_{i}\).
  If \(\tilde\pi\) is obtained from~\(\pi\) by adding the assignment~\(1\mapsto i\), we get
  \begin{align}
    \FF(d^{1}_{i}\,\CobarEl{x};u_{\pi}) &= \FF(\CobarEl{d_{i} x};u_{\pi}) = \Sz_{\pi} d_{i}\,x = d_{0}\,\Sz_{\tilde\pi}\,x \\
    \notag &= d_{0}\,\FF(\CobarEl{x};u_{\tilde\pi}) = \FF(\CobarEl{x};d_{0}\,u_{\tilde\pi}) = \FF\bigl(\CobarEl{x};(\delta^{1}_{i})_{*}\,u_{\pi}\bigr),
  \end{align}
  where we have used \Cref{thm:properties-sz-new-d} and \Cref{thm:d-u-tilde-pi}.

  We now come to the case~\(\lambda=\delta^{0}_{i}\).
  Under the bijection from \Cref{thm:Psi-k-l}, the permutation~\(\pi\in S_{n-1}\)
  corresponds to~\(((\alpha,\beta),\sigma,\tau)\), where \(\sigma\in S_{i-1}\), \(\tau\in S_{n-i}\)
  and \((\alpha,\beta)\) is an \((i-1,n-i)\)-shuffle.
  The decomposition~\(\II^{n-1}=\II^{i-1}\times\II^{n-i}\) gives a decomposition of~\(u_{\pi}\)
  into the \((n-1)\)-simplices~\(u_{1}\) and~\(u_{2}\). By \Cref{thm:decomposition-simplex-product-cube}, we have
  \begin{equation}
    u_{1} = s_{\beta-1}\,u_{\sigma}
    \qquad\text{and}\qquad
    u_{2} = s_{\alpha-1}\,u_{\tau}.
  \end{equation}
  Using these identities together with \Cref{thm:properties-sz-new-d} and \Cref{thm:d-u-tilde-pi} as before gives
  \begin{align}
    \FF(d^{0}_{i}\,\CobarEl{x};u_{\pi}) &= \FF(\CobarEl{x(0\dots i),x(i\dots n+1)};(u_{1},u_{2})) \\
    \notag &= \FF\bigl(\CobarEl{x(0\dots i)};u_{1}\bigr)\cdot \FF\bigl(\CobarEl{x(i\dots n+1)};u_{2}\bigr) \\
    \notag &= \FF\bigl(\CobarEl{x(0\dots i)};s_{\beta-1}\,u_{\sigma}\bigr)\cdot \FF\bigl(\CobarEl{x(i\dots n+1)};s_{\alpha-1}\,u_{\tau}\bigr) \\
    \notag &= s_{\beta-1}\Sz_{\sigma}x(0\dots i) \cdot s_{\alpha-1}\Sz_{\tau}x(i\dots n+1) \\
    \notag &= d_{n} \Sz_{\tilde\pi} x = d_{n} \FF(\CobarEl{x};u_{\tilde\pi}) = \FF(\CobarEl{x};d_{n}\,u_{\tilde\pi}) \\
    \notag &= \FF\bigl(\CobarEl{x};(\delta^{0}_{i})_{*}\,u_{\pi}\bigr).
  \end{align}

  We now look at the case~\(\lambda=\gamma_{i}\),  still assuming \(I=\emptyset\) and~\(u=u_{\pi}\).
  We set \(j=\min(\pi^{-1}(i),\pi^{-1}(i+1))\), and we let \(\tilde\pi\) be obtained from~\(\pi\)
  by removing the assignment~\(j\mapsto\pi(j)\). This time we get
  \begin{align}
    \FF(\gamma_{i}\,\CobarEl{x};u_{\pi}) &= \FF(\CobarEl{s_{i}\,x};u_{\pi}) = \Sz_{\pi} s_{i}\,x = s_{j-1}\,\Sz_{\tilde\pi}\,x \\
    \notag &= s_{j-1}\,\FF(\CobarEl{x};u_{\tilde\pi}) = \FF\bigl(\CobarEl{x};s_{j-1}\,u_{\tilde\pi}\bigr) = \FF\bigl(\CobarEl{x};(\gamma_{i})_{*}u_{\pi}\bigr).
  \end{align}

  The extension to arbitrary simplices~\(u\in\II_{\CobarEl{x}}\) is done exactly as in the proof of \Cref{thm:F-s-CobarEl-x-u},
  still under the assumption~\(I=\emptyset\). The extension to an arbitrary inner interval~\(I\) is also analogous
  to the case~\(\lambda=s_{i}\) considered in \Cref{thm:F-s-CobarEl-x-u}.
  In fact, by the transposes of cubical relations or the definition of the morphisms in the cube category,
  we have \(\sigma_{I}\lambda=\lambda'\sigma_{I'}\) for some inner subset~\(I'\)
  and some morphism~\(\lambda'\) which is either some~\(\delta^{\epsilon}_{i'}\), some~\(\gamma_{i'}\) or the identity map.
  Hence the argument presented for \Cref{thm:F-s-CobarEl-x-u} carries over to the present context.
  This completes the proof.
\end{proof}

For~\(\lambda=\delta^{\epsilon}_{i}\),~\(\sigma_{i}\) or~\(\gamma_{i}\),
the identity~\eqref{eq:F-lambda-z-u} now follows for products by the multiplicativity of~\(\FF\).
If \(z=z_{1}z_{2}\), then \(\lambda^{*}z = (\lambda_{1}^{*}z_{1})(\lambda_{2}^{*}z_{2})\)
for some morphisms~\(\lambda_{1}\) and~\(\lambda_{2}\) in the cube category,
one of which is of the same kind as~\(\lambda\) and the other one the identity map. Hence
if \eqref{eq:F-lambda-z-u} holds for pairs triples~\((z_{1},\lambda_{1})\),~\((z_{2},\lambda_{2})\)
and all simplices~\(u_{1}\),~\(u_{2}\), then we also have for a simplex~\(u=(u_{1},u_{2})\)
\begin{align}
  \FF(\lambda^{*}z;u) &= \FF\bigl(\lambda_{1}^{*}z_{1},\lambda_{2}^{*}z_{2};(u_{1},u_{2})\bigr)
  = \FF(\lambda_{1}^{*}z_{1};u_{1})\,\FF(\lambda_{2}^{*}z_{2};u_{2}) \\
  \notag &= \FF(z_{1};(\lambda_{1})_{*}u_{1})\,\FF(z_{2};(\lambda_{2})_{*}u_{2}) \\
  \notag &= \FF\bigl(z_{1},z_{2};(\lambda_{1},\lambda_{2})_{*}(u_{1},u_{2})\bigr) = \FF(z;\lambda_{*}u).
\end{align}
Because all elements of~\(\bbOM X\) are products of elements of the form~\(z=\CobarEl{x}_{I}\),
this establishes \eqref{eq:F-lambda-z-u} for all~\(z\in\bbOM X\) and all simplices~\(u\)
provided that \(\lambda=\delta^{\epsilon}_{i}\),~\(\sigma_{i}\) or~\(\gamma_{i}\).
(For the product \(z=\CobarEl{}\) of length~\(0\), any \(m\)-simplex~\(u\) is sent to~\(1\in G_{m}\) anyway.)

The case of general~\(\lambda\) now follows easily: If the identity~\eqref{eq:F-lambda-z-u} holds for~\(\lambda\) and~\(\mu\),
then it holds for~\(\lambda\,\mu\) since
\begin{align}
  \FF((\lambda\,\mu)^{*}z;u) &= \FF(\mu^{*}\lambda^{*}z;u) = \FF(\lambda^{*}z;\mu_{*}u) \\
  \notag &= \FF(z;\lambda_{*}\mu_{*}u) = \FF(z;(\lambda\,\mu)_{*}u).
\end{align}
As the morphisms in the cube category are generated by the maps~\(\sigma_{i}\),~\(\delta^{\epsilon}_{i}\) and~\(\gamma_{i}\),
this shows that we get a well-defined simplicial map~\(f\colon\TT\bbOM X\to G\).
That \(f\) is multiplicative follows from the definition~\eqref{thm:def-F-product}
together with the compatibility of the triangulation functor with products of cubical sets
(\Cref{thm:triangulation-product}).
This completes the proof of part~\ref{thm:main-1} of \Cref{thm:main}.

\medbreak

We finally come to part~\ref{thm:main-2}. The composition
\begin{equation}
  \OM\,C(X) \cong C(\bbOM X) \stackrel{\tt}{\longrightarrow} C(\TT \bbOM X) \stackrel{C(\FFF)}{\longrightarrow} C(G),
\end{equation}
is determined by the assignment
\begin{equation}
  \CobarEl{x} \mapsto \sum_{\pi\in S_{n}} (-1)^{\deg{\pi}}\Sz_{\pi}x
\end{equation}
for~\(x\in X_{n+1}\). A look at~\eqref{eq:szczarba-twisting-cochain} shows
that this is the same as the formula~\eqref{eq:Omega-CX-CG} for the Szczarba map;
the signs agree by \Cref{thm:parity}. Note that if \(x\) is a \(1\)-simplex,
then \(\CobarEl{x}\) is sent to~\(0\) by both maps since we assume \(X\) to be \(1\)-reduced.
This establishes part~\ref{thm:main-2} of \Cref{thm:main} and completes the proof.

\end{document}